\documentclass[11pt,a4paper]{amsart}
\usepackage{amsmath,amsfonts,amsthm,amssymb,nameref}
\usepackage[top=3cm, bottom=2.8cm, left=3cm, right=3cm]{geometry} %page margins

\usepackage{tikz} %makes graphs
\usetikzlibrary{calc,automata} %lets tikz calculate coordinates,enables nodes

\usepackage{color} %we like drawings

\usepackage{enumerate}

\usepackage[pdftex]{hyperref}
\hypersetup{colorlinks,citecolor=black,filecolor=black,linkcolor=black,urlcolor=black}
\usepackage[all]{hypcap}
\usepackage{microtype} % combats overfull boxes

\newcommand{\freeproduct} {*}
\newcommand{\N}	{\mathbb N}
\newcommand{\Z}	{\mathbb Z}
\newcommand{\R}	{\mathbb R}

\newcommand{\Cay}	{\operatorname{Cay}}
\newcommand{\diam}	{\operatorname{diam}} %for diameter

\newcommand{\Div}	{\operatorname{Div}}

\newcommand{\Area}	{\operatorname{Area}_{\mathrm{rel}}}

\newtheorem{thm}{Theorem}[section] 
\newtheorem{prop}[thm]{Proposition}
\newtheorem{lem} [thm]{Lemma} 
\newtheorem{cor} [thm]{Corollary}
\newtheorem*{thm*} {Theorem} 
\newtheorem*{prop*}{Proposition}
\newtheorem*{lem*} {Lemma} 
\newtheorem*{cor*} {Corollary}
\newtheorem{innercustomthm}{Theorem}
\newenvironment{customthm}[1]
  {\renewcommand\theinnercustomthm{#1}\innercustomthm}
  {\endinnercustomthm}
\newtheorem{innercustomprop}{Proposition}
\newenvironment{customprop}[1]
  {\renewcommand\theinnercustomprop{#1}\innercustomprop}
  {\endinnercustomprop}
\theoremstyle{definition}

\newtheorem{defi}[thm]{Definition}
\newtheorem{example}[thm]{Example}
\newtheorem{remark}[thm]{Remark}
\newtheorem*{defi*}{Definition}
\newtheorem*{example*}{Example}
\newtheorem*{remark*}{Remark}
\newtheorem*{problem*}{Problem}
\newtheorem*{convention*}{Convention}
\newtheorem*{defi1*}{Definition I of the WPD element}
\newtheorem*{defi2*}{Definition II of the WPD element}

\begin{document}
\title[Small cancellation groups are acylindrically hyperbolic]{Infinitely presented graphical small cancellation groups are acylindrically hyperbolic}
\subjclass[2010]{Primary: 20F06; Secondary: 20F65, 20F67.}
\author{Dominik Gruber}
 \thanks{The first author's work was supported by the ERC grant no.\ 259527 of Goulnara Arzhantseva, by a dissertation completion fellowship 2015 of the University of Vienna, and by the Swiss National Science Foundation Professorship FN PP00P2-144681/1 of Laura Ciobanu.}
 \address{Institut de Math\'ematiques, Universit\'e de Neuch\^atel, Rue Emile-Argand 11, 2000 Neuch\^atel, Switzerland}
 \email{dominik.gruber@unine.ch}
\author{Alessandro Sisto}
\address{Department of Mathematics, ETH Zurich, 8092 Zurich, Switzerland}
 \email{sisto@math.etzh.ch}

 \begin{abstract}
We prove that infinitely presented graphical $Gr(7)$ small cancellation groups are acylindrically hyperbolic. In particular, infinitely presented classical $C(7)$-groups and, hence, classical $C'(\frac{1}{6})$-groups are acylindrically hyperbolic. We also prove the analogous statements for the larger class of graphical small cancellation presentations over free products. We construct infinitely presented classical $C'(\frac{1}{6})$-groups that provide new examples of divergence functions of groups.
\end{abstract}

\maketitle

\section{Introduction}

Small cancellation theory is a rich source of finitely generated infinite groups with exotic or extreme properties. \emph{Graphical} small cancellation theory is a generalization of classical small cancellation theory introduced by Gromov \cite{Gro}. It is a tool for constructing groups with prescribed embedded subgraphs in their Cayley graphs. Groups containing expander graphs in a suitable way have provided the only known counterexamples to the Baum-Connes conjecture with coefficients \cite{HLS}. Gromov proved the existence of such so-called \emph{Gromov's monsters} in an intricate geometric construction \cite{Gro} which was then explained in detail by Arzhantseva-Delzant \cite{ADe}. Ollivier \cite{Oll} and, subsequently, the first author \cite{Gru} proved that combinatorial interpretations of Gromov's graphical small cancellation theory provide more elementary tools for producing new examples of Gromov's monsters. Following this, Osajda recently showed, by probabilistic arguments, the existence of suitable small cancellation labellings of certain expander graphs \cite{Osa-label}, thus completing this combinatorial construction.

The main result of this article is that infinitely presented graphical $Gr(7)$-groups are acylindrically hyperbolic. The $Gr(7)$ small cancellation condition for labelled graphs was shown to produce lacunary hyperbolic groups with coarsely embedded prescribed infinite sequences of finite graphs by the first author \cite{Gru}. Our theorem, in particular, applies to the new Gromov's monsters mentioned above. Moreover, it applies to the only known non-coarsely amenable groups with the Haagerup property of Arzhantseva-Osajda \cite{AO2,Osa-label}. Finally, our theorem also covers all examples of infinitely presented classical $C(7)$-groups as defined in \cite[Chapter V]{LS}. 

Our result, moreover, extends to groups given by infinite graphical small cancellation presentations over free products of groups. Such groups were used to prove various embedding theorems \cite[Chapter V]{LS} and, more recently, to construct examples of torsion-free hyperbolic groups and direct limits of torsion-free hyperbolic groups without the unique product property \cite{Ste,Arzhantseva-Steenbock}.

\subsection*{Acylindrical hyperbolicity} A group is \emph{acylindrically hyperbolic} if it is non-elementary and admits an acylindrical action with unbounded orbits on a Gromov-hyperbolic space, see Definition \ref{defi:acylindrical}. This definition of Osin \cite{Os-acyl} unified several equivalent far-reaching generalizations of relative hyperbolicity \cite{BeFu-wpd,Ha-isomhyp, DGO, Si-contr}. Acylindrical hyperbolicity has strong consequences: Every acylindrically hyperbolic group is SQ-universal, it contains free normal subgroups \cite{DGO}, it contains Morse elements and hence all its asymptotic cones have cut-points \cite{Si-hypembmorse}, and its bounded cohomology is infinite dimensional in degrees 2 \cite{HullOsin} and 3 \cite{FPS-H3b-acylhyp}. Moreover, if an acylindrically hyperbolic group does not contain finite normal subgroups, then its reduced $C^*$-algebra is simple \cite{DGO} and every commensurating endomorphism is an inner automorphism \cite{MOS-pwise-inner}. 

Since our theorem, in particular, applies to all infinitely presented classical $C'(\frac{1}{6})$-groups, we deduce that the infinite groups without non-trivial finite quotients constructed by Pride \cite{Pri} are not simple because they are SQ-universal and thus have uncountably many proper quotients, and that the groups with two non-homeomorphic asymptotic cones due to Thomas-Velickovic \cite{TV} have cut-points in all of their asymptotic cones. 

Beyond these straightforward corollaries, our method of proof also provides actions on very concrete Gromov hyperbolic spaces for the groups under consideration and, hence, a strong tool for studying them further.

\subsection*{Divergence} The class of acylindrically hyperbolic groups is extensive: It contains all non-elementary hyperbolic and relatively hyperbolic groups, mapping class groups, $\mathrm{Out}(F_n)$, the Cremona group in dimension 2, many $\mathrm{CAT}(0)$ groups and many groups acting on trees, see \cite{DGO,MO} and references therein. 
It is unknown whether the class of acylindrically hyperbolic groups is closed under quasi-isometries of groups \cite[Problems 9.1, 9.2]{DGO}. Analyzing how given quasi-isometry invariants behave in this class of groups is a way to shed light on this question. Therefore, in the second part of the paper, we consider a quasi-isometry invariant called \emph{divergence}. 

The divergence function of a 1-ended group measures the lengths of paths between two points that avoid given balls in the Cayley graph. It was first studied in \cite{Gr-asinv} and \cite{Ger-div}, and in recent years, for example, in \cite{Be-asgeommcg,OOS-lacun,DR-divteich,DMS-div,BC-shortconj,BC-raagcones, Si-metrrh}. 
Acylindrically hyperbolic groups have superlinear divergence \cite{Si-hypembmorse}, as they contain Morse elements. For non-elementary hyperbolic groups the divergence functions are exponential, while for mapping class groups of closed oriented surfaces of genus at least 2 they are quadratic \cite{Be-asgeommcg, DR-divteich}. For every degree integer $d\geqslant 1$, there exists a CAT(0)-group with divergence function equivalent to $x^d$ \cite{BC-shortconj,Ma}. In general, however, the problem of determining which functions can be obtained as divergence functions of groups is wide open.

We show that such tame behavior as in the mentioned examples cannot be expected in general, even in the class of acylindrically hyperbolic groups: For any given countable set of subexponential functions, we provide an explicit infinitely presented classical $C'(\frac{1}{6})$-group with a divergence function whose limit superior exceeds each of the chosen subexponential functions, while its limit inferior is bounded by a quadratic polynomial.
Previously, no example of a divergence function of a finitely generated group having a behavior of this type was known.

\par\medskip

We conclude by providing a tool that enables constructions of groups that are not non-trivially relatively hyperbolic. We use this tool and a small cancellation construction over free products to show that every finitely generated infinite group is a non-degenerate hyperbolically embedded subgroup of a finitely generated non-relatively hyperbolic group.

\subsection{Statement of results}

We state our main result on the acylindrical hyperbolicity of graphical $Gr(7)$-groups. For precise definitions of the graphical small cancellation conditions we use, see Section~\ref{section:preliminaries}. 

Given a labelled graph $\Gamma$, let $G(\Gamma)$ be the group whose generating set is the (possibly infinite) set of labels and whose relators are the words read on closed paths in $\Gamma$. A \emph{piece} is a labelled path that occurs in two distinct places in $\Gamma$, where distinct means distinct up to automorphisms of $\Gamma$. The graphical versions of the $C(7)$ small cancellation condition require that no non-trivial closed path is made up of fewer than 7 pieces. We can interpret any set of relators of a group presentation as a labelled graph that consists of disjoint cycle graphs labelled by the relators. Thus, graphical small cancellation conditions generalize classical small cancellation conditions.

We consider two variants of graphical generalizations of the classical $C(7)$-condition: 
The \emph{graphical $Gr(7)$-condition} allows non-trivial label-pre\-ser\-ving automorphisms of $\Gamma$ and fully generalizes the classical $C(7)$-condition, i.e.\ all classical $C(7)$-presentations are graphical $Gr(7)$-presentations, see \cite[Section~1.1]{Gru}. In particular, it allows constructions of groups with torsion. Since any group has a $Gr(7)$-presentation given by its labelled Cayley graph, see \cite[Example~2.2]{Gru}, we always have to make additional assumptions on the graph to obtain meaningful statements. The \emph{graphical $C(7)$-condition}, on the other hand, does not allow non-trivial label-preserving automorphisms of $\Gamma$ and always yields torsion-free (and, in fact, 2-dimensional) groups, see \cite[Theorem~2.18]{Gru}. %All classical $C(7)$-presentations where no relators are proper powers are graphical $C(7)$-presentations.
The following are our main results:

\begin{thm}\label{thm:gr7} Let $\Gamma$ be a $Gr(7)$-labelled graph whose components are finite. Then $G(\Gamma)$ is either virtually cyclic or acylindrically hyperbolic. 
\end{thm}

In most cases, $G(\Gamma)$ is not virtually cyclic, see Remark~\ref{remark:non-elementary}.

\begin{thm}\label{thm:c7} Let $\Gamma$ be a $C(7)$-labelled graph. Then $G(\Gamma)$ is either trivial, infinite cyclic, or acylindrically hyperbolic.
\end{thm}

Since every component of $\Gamma$ embeds into the Cayley graph of $G(\Gamma)$ by \cite[Lemma~4.1]{Gru}, this group can only be trivial if every component of $\Gamma$ has at most one vertex. It can only be infinite cyclic if every simple closed path on $\Gamma$ has length at most 2 by \cite[Remark 3.3]{Gru}.

To prove acylindrical hyperbolicity, we construct a hyperbolic space $Y$ on which $G(\Gamma)$ acts. It is  obtained from the Cayley graph of $G(\Gamma)$ by coning-off the embedded copies of components of $\Gamma$. In the classical $C(7)$-case, these components are cycle graphs labelled by the relators. In fact, our construction of a hyperbolic space works in a more general setting  by coning-off relators of an infinite presentations satisfying a certain (relative) subquadratic isoperimetric inequality. 

Let $M(S)$ denote the free monoid on $S\sqcup S^{-1}$. Given $w\in M(S)$, a \emph{diagram for $w$} is a singular disk diagram with boundary word $w$, see Subsection~\ref{subsection:diagrams}.

\begin{customprop}{\ref{prop:hyperbolic_space}} Let $\langle S\mid R\rangle$ be a presentation of a group $G$,  %where $S\subseteq G$ and 
where $R\subseteq M(S)$ is closed under cyclic conjugation and inversion. Let $W_0$ be the set of all subwords of elements of $R$. Suppose there exists a subquadratic map $f:\N\to\N$ with the following property for every $w\in M(S)$: If $w$ is trivial in $G$ and if $w$ can be written as product of $N$ elements of $W_0$, then there exists a diagram for $w$ over $\langle S\mid R\rangle$ with at most $f(N)$ faces. Denote by $W$ the image of $W_0$ in $G$. Then $\Cay(G,S\cup W)$ is Gromov hyperbolic. 
\end{customprop}

For the action of $G(\Gamma)$ on $Y$ we construct a particular type of loxodromic element, called WPD element, see Definition~\ref{WPDdefn}. It is defined as a suitable product of labels of paths on $\Gamma$. In the classical $C(7)$-case, these labels of paths are subwords of relators. In fact, our proof of Theorem~\ref{thm:gr7} does not require all components of $\Gamma$ to be finite, only some necessary to find the WPD element. The existence of an action of $G(\Gamma)$ on a hyperbolic space with a WPD element implies that $G(\Gamma)$ is virtually cyclic or acylindrically hyperbolic. In the case that $\Gamma$ satisfies the stronger graphical $Gr'(\frac{1}{6})$-condition, Theorem~\ref{thm:gr16} relaxes the finiteness assumption of Theorem~\ref{thm:gr7}, and Proposition~\ref{prop:geodesics} provides a description of geodesics in the space $Y$ that allows us to construct an example where the action of $G(\Gamma)$ on $Y$ is not acylindrical.

Using a new viewpoint on graphical small cancellation theory over free products presented in \cite{Gru-SQ}, we show that our results immediately generalize to the larger class of groups given by graphical small cancellation presentations over free products, see Theorems~\ref{thm:gr7*},~\ref{thm:c7*}, and~\ref{thm:gr16*}.

Since every classical $C(7)$-presentation corresponds to a $Gr(7)$-labelled graph whose components are cycle graphs labelled by the relators, Theorem~\ref{thm:gr7} implies the following:

\begin{thm}\label{classical}
Let $G=\langle S\mid R\rangle$ be a classical $C(7)$-presentation. Then $G$ is either virtually cyclic or acylindrically hyperbolic.% If $S$ is finite and $R$ is infinite, then $G$ is acylindrically hyperbolic.
\end{thm}

Since every classical $C'(\frac{1}{6})$-presentation is a classical $C(7)$-presentation, this result, in particular, holds for classical $C'(\frac{1}{6})$-groups.

\begin{cor}\label{classical_corollary}
 Let $G=\langle S\mid R\rangle$ be a classical $C(7)$-presentation, where $S$ is finite and $R$ is infinite. Then $G$ is acylindrically hyperbolic, whence:
 \begin{itemize}
  \item $G$ is SQ-universal, i.e. for every countable group $C$ there exists a quotient $Q$ of $G$ such that $C$ embeds into $Q$.
  \item $G$ has free normal subgroups.
  \item $G$ contains Morse elements, and all its asymptotic cones have cut-points.
  \item The bounded cohomology of $G$ is infinite dimensional in degrees 2 and 3.
  \item The reduced $C^*$-algebra of $G$ is simple.
  \item Every commensurating endomorphism of $G$ is an inner automorphism.
 \end{itemize}
\end{cor}

The first statement of Corollary~\ref{classical_corollary} was proven directly using different arguments by the fist author for the larger class of infinitely presented classical $C(6)$-groups \cite{Gru-SQ}.

We next state our result on divergence functions of groups. The group presentations we obtain are explicit if the subexponential functions are given explicitly.

\begin{customthm}{\ref{thm:divalt}}%\label{intro:divergence}
 Let $r_N:=(a^{N}b^{N}a^{-N}b^{-N})^4$, and for $I\subseteq \N$, let $G(I)$ be defined by the presentation $\langle a,b\mid r_i:i\in I\rangle$. Then, for every infinite set $I$, we have:
\begin{equation*}
\liminf_{n\to\infty} \frac{\Div^{G(I)}(n)}{n^2}<\infty.
\end{equation*}

 Let $\{f_k\mid k\in \N\}$ be a countable set of subexponential functions. Then there exists an infinite set $J\subseteq \N$ such that for every function $g$ satisfying $g\preceq f_k$ for some $k$ we have for every subset $I\subseteq J$:
 
 \begin{equation*}
\limsup_{n\to\infty}\frac{\Div^{G(I)}(n)}{g(n)}=\infty.
\end{equation*}

\end{customthm}

Being acylindrically hyperbolic is equivalent to containing a proper infinite \emph{hyperbolically embedded} subgroup in the sense of \cite{DGO}, and the motivating examples of hyperbolically embedded subgroups are peripheral subgroups of relatively hyperbolic groups. In the last part of the paper we explore the relation between these notions using our small cancellation constructions.

In Proposition~\ref{notrh}, we provide a condition that enables constructions of groups that are not non-trivially relatively hyperbolic. In particular, the groups of Theorem~\ref{thm:divalt} are not non-trivially relatively hyperbolic, which can also be deduced from the fact that the divergence function of any non-trivially relatively hyperbolic group is at least exponential \cite{Si-metrrh}. We use Proposition~\ref{notrh} and explicit small cancellation presentations over free products to prove the following theorem, which provides new examples of non-degenerate hyperbolically embedded subgroups of non-relatively hyperbolic groups.

\begin{customthm}{\ref{thm:hyperbolically_embedded}} Let $H$ be a finitely generated infinite group. Then there exists a finitely generated group $G$ such that $H$ is a non-degenerate hyperbolically embedded subgroup of $G$ and such that $G$ is not hyperbolic relative to any collection of proper subgroups. 
\end{customthm}

\section{Preliminaries}\label{section:preliminaries}

In this section, we give definitions and useful preliminary results.

\subsection{Acylindrical hyperbolicity via WPD-condition}

To prove acylindrical hyperbolicity of the groups we consider, we will use the following equivalent definition given in \cite{Os-acyl}. Here, WPD is an abbreviation for ``weak proper discontinuity'' as defined in \cite{BeFu-wpd}.

\begin{defi}\label{defi:wpd}
\label{WPDdefn}
 Let $G$ be a group acting by isometries on a Gromov hyperbolic space $Y$. We say $g\in G$ is a \emph{WPD element} if both of the following hold:
 \begin{itemize}
  \item $g$ acts \emph{hyperbolically}, i.e.\ for every $x\in Y$, the map $\Z\to Y, z\mapsto g^zx$ is a quasi-isometric embedding, and
  \item $g$ satisfies the \emph{WPD condition}, i.e.\ for every $x\in Y$ and every $K\geqslant 0$ there exists $N_0\geqslant 0$ such that for all $N\geqslant N_0$ the following set is finite:
   $$\{h\in G\mid d_Y(x,hx)\leqslant K\text{ and }d_Y(g^Nx,hg^Nx)\leqslant K\}.$$
 \end{itemize}
 We say $G$ is \emph{acylindrically hyperbolic} if $G$ is not virtually cyclic and if there exists an action of $G$ by isometries on a Gromov hyperbolic space for which there exists a WPD element.
\end{defi}

\subsection{Graphical small cancellation conditions}\label{subsection:definitions_graphical} We give definitions of graphical small cancellation conditions following \cite{Gru,Gru-SQ}.

Let $\Gamma$ be a graph. A \emph{labelling} of $\Gamma$ by a set $S$ is a choice of orientation for each edge and a map assigning to each edge an element of $S$. Given an edge-path $p$ in $\Gamma$, the \emph{label} of $p$, denoted $\ell(p)$, is the product of the labels of the edges traversed by $p$ in $M(S)$, the free monoid on $S\sqcup S^{-1}$. Here a letter is given exponent $+1$ if the corresponding edge is traversed in its direction and $-1$ if it is traversed in the opposite direction. A labelling is \emph{reduced} if the labels of reduced paths are freely reduced words. Here, a reduced path is a path without backtracking. The \emph{group defined by} $\Gamma$ is given by the following presentation:
$$G(\Gamma):=\langle S\mid\text{labels of simple closed paths in }\Gamma\rangle.$$

If $p$ is an oriented edge or a path, then $\iota p$ denotes its initial vertex and $\tau p$ its terminal vertex. Given graphs $\Theta$ and $\Gamma$ that are labelled by the same set $S$ and a path $p$ in $\Theta$, a \emph{lift} of $p$ is a path in $\Gamma$ that has the same label as $p$. Given a labelled graph $\Gamma$ and paths (or subgraphs) $p_1$ and $p_2$ in $\Gamma$, we say $p_1$ and $p_2$ are \emph{essentially distinct} if, for every label-preserving automorphism $\phi$ of $\Gamma$, we have $p_2\neq\phi(p_1)$. Otherwise, we say they are \emph{essentially equal}.

\begin{defi}[Piece]\label{defi:piece} Let $\Theta$ and $\Gamma$ be graphs labelled over the same set $S$. A \emph{piece} in $\Theta$ with respect to $\Gamma$ is a path $p$ in $\Theta$ for which there exist essentially distinct lifts $p_1$ and $p_2$ of $p$ in $\Gamma$.
\end{defi}

In most cases we consider, $\Theta$ will either be $\Gamma$ itself or the 1-skeleton of a diagram as defined below. We are ready to give the graphical small cancellation conditions first stated in \cite{Gru}. A path is \emph{non-trivial} if it is not 0-homotopic. %The graphical $Gr(7)$-condition is due to Gromov \cite{Gro}, and the graphical $C(7)$-condition was subsequently introduced by Ollivier \cite{Oll}.
\begin{defi} Let $\Gamma$ be a labelled graph, $n\in\N$, and $\lambda>0$. We say $\Gamma$ satisfies
\begin{itemize}
 \item the \emph{graphical $Gr(n)$-condition} if the labelling is reduced and no non-trivial closed path is a concatenation of fewer than $n$ pieces,
 \item the \emph{graphical $Gr'(\lambda)$-condition} if the labelling is reduced and every piece $p$ that is a subpath of a simple closed path $\gamma$ satisfies $|p|<\lambda|\gamma|$.  
\end{itemize}
If, moreover, every label-preserving automorphism of $\Gamma$ is the identity on every component with non-trivial fundamental group, then we say $\Gamma$ satisfies the graphical $C(n)$-condition, respectively graphical $C'(\lambda)$-condition.
\end{defi}
Note that the $Gr'(\frac{1}{n})$-condition is stronger than the $Gr(n+1)$-condition. %(and, in fact, strictly stronger by \cite[Theorem~4.6]{Gru-SQ}).

\subsection{Diagrams over graphical presentations}\label{subsection:diagrams}

We briefly give definitions of diagrams in the sense of \cite{LS}, which are standard tools in small cancellation theory.

A \emph{diagram} is a finite, contractible 2-complex embedded in $\R^2$ whose 1-skeleton, when considered as graph, is endowed with a labelling by a set $S$. (Such a diagram is usually called a singular disk diagram.) A \emph{disk diagram} is a diagram that is homeomorphic to a 2-disk. We call the images of 1-cells \emph{edges} and the images of 2-cells \emph{faces}.

The \emph{boundary word} of a face $\Pi$ is the word read on the boundary cycle $\partial \Pi$ of $\Pi$ (which depends on choices of base point and orientation). The \emph{boundary word} of a diagram $D$ is the word read on the path $\partial D$ traversing the topological boundary of $D$ inside $\R^2$ (and again depends on choices of base point and orientation). We call a diagram $D$ with boundary word $w$ a \emph{diagram for $w$}. Given a presentation $\langle S\mid R\rangle$ where $R$ is a set of words in $M(S)$, the free monoid on $S\sqcup S^{-1}$, a diagram \emph{over} $\langle S\mid R\rangle$ is a diagram labelled by $S$ where every face has a boundary word in $R$.

An \emph{arc} in a diagram $D$ is an embedded line graph of length at least 1 all interior vertices of which have degree 2 (in $D$). An arc is \emph{exterior} if its edges lie in $\partial D$ and \emph{interior} otherwise. 
A face in a diagram $D$ is called \emph{interior} if it does not have edges in the boundary of $D$, otherwise it is a \emph{boundary face}. A \emph{$(3,k)$-diagram} is a diagram where the boundary of every interior face is made up of at least $k$ maximal arcs (where we count arcs with their multiplicity in the boundary cycle).

Given an $S$-labelled graph, a \emph{diagram over $\Gamma$} is a diagram over $\langle S\mid\text{labels of simple}$ $\text{closed paths in }\Gamma\rangle$. Let $p$ be path in a diagram over $\Gamma$ that lies in the intersection of faces $\Pi$ and $\Pi'$. There are lifts of $p$ in $\Gamma$ induced by lifts of $\partial \Pi$ and $\partial \Pi'$ in $\Gamma$. We say that $p$ \emph{originates} from $\Gamma$ if there exist lifts of $\partial \Pi$ and $\partial \Pi'$ in $\Gamma$ that induce the same lift of $p$. Note that if an interior arc does not originate from $\Gamma$, then it is a piece. A \emph{$\Gamma$-reduced} diagram is a diagram over $\Gamma$ where no interior edge (or equivalently no interior arc) originates from $\Gamma$. The following lemma, a direct consequence of van Kampen's Lemma \cite{LS} and \cite[Lemma 2.13]{Gru}, shows that this is a notion of ``reducedness'' suitable for our purposes . The idea is that whenever an arc $a$ in the intersection of two faces originates from $\Gamma$, then the two incident faces can be merged into one face by removing $a$, and the resulting face bears the label of a closed path in $\Gamma$. This closed path can be decomposed into simple closed paths, which corresponds to foldings in the diagram.

\begin{lem}\label{lem:graphical_basic} Let $\Gamma$ be a $Gr(6)$-labelled graph, and let $w$ be a word in $M(S)$. Then $w$ represents the identity in $G(\Gamma)$ if and only if there exists a $\Gamma$-reduced diagram for $w$.
\end{lem}

Consider a $\Gamma$-reduced diagram $D$, where $\Gamma$ is $Gr(k)$-labelled for $k\geqslant 6$. Whenever an arc $a$ lies in the intersection of two faces $\Pi$ and $\Pi'$ of $D$, it is a piece. Therefore, $D$ is a $(3,k)$-diagram. If $\Gamma$ moreover satisfies the graphical $Gr'(\lambda)$-condition, then we have $|a|<\lambda|\partial\Pi|$ and $|a|<\lambda|\partial\Pi'|$.

%If $S$ and $\Gamma$ are finite and $\Gamma$ satisfies the graphical $Gr(7)$-condition, then $G(\Gamma)$ is hyperbolic \cite[Theorem 2.16]{Gru}. 

\subsection{Graphical small cancellation over free products}\label{subsection:definitions_products}

Given groups $G_i$ with generating sets $S_i$, let $\Gamma$ be a graph labelled by the set $\sqcup_{i\in I}S_i$. In this situation, we define $G(\Gamma)_*$ to be the quotient of $\freeproduct_{i\in I}G_i$ by the normal subgroup generated by all labels of closed paths in $\Gamma$. 

Graphical small cancellation groups over free products were first studied in \cite{Ste}. We recap here definitions from \cite{Gru-SQ}, which present a convenient way to skip notions such as ``reduced forms'' and ``semi-reduced forms'' used in standard definitions of small cancellation conditions over free products \cite{LS}.

\begin{defi} Let $\Gamma$ be labelled by $\sqcup_{i\in I}S_i$, where $S_i$ are generating sets of groups $G_i$. The \emph{completion} of $\Gamma$, denoted $\overline \Gamma$, is obtained as follows: Onto every edge labelled by $s_i\in S_i$, attach a copy of $\Cay(G_i,S_i)$ along an edge of $\Cay(G_i,S_i)$ labelled by $s_i$. The graph $\overline \Gamma$ is defined as the quotient of the resulting graph by the equivalence relation $e\sim e'$ if edges $e$ and $e'$ have the same label and there exists a path from $\iota e$ to $\iota e'$ whose label is trivial in $\freeproduct_{i\in I}G_i$. 
\end{defi}

We use the same notion of piece as above. We say a path $p$ in $\overline \Gamma$ is \emph{locally geodesic} if every subpath of $p$ that is contained in one of the attached $\Cay(G_i,S_i)$ is geodesic. More generally, a path in a labelled graph is locally geodesic if it lifts to a locally geodesic path in $\overline\Gamma$.

\begin{defi}  Let $n\in\N$ and $\lambda>0$. Let $\Gamma$ be labelled over $\sqcup_{i\in I}S_i$, where $S_i$ are generating sets of groups $G_i$. We say $\Gamma$ satisfies
\begin{itemize}
 \item the \emph{graphical $Gr_*(n)$-condition} if every attached $\Cay(G_i,S_i)$ in $\overline \Gamma$ is an embedded copy of $\Cay(G_i,S_i)$ and in $\overline \Gamma$ no path whose label is non-trivial in $\freeproduct_{i\in I}G_i$ is concatenation of fewer than $n$ pieces,
 \item the \emph{graphical $Gr_*'(\lambda)$-condition} if every attached $\Cay(G_i,S_i)$ in $\overline \Gamma$ is an embedded copy of $\Cay(G_i,S_i)$ and in $\overline \Gamma$ every piece $p$ that is locally geodesic and that is a subpath of a simple closed path $\gamma$ whose label is non-trivial in $\freeproduct_{i\in I}G_i$ satisfies $|p|<\lambda|\gamma|$.  
\end{itemize}
If, additionally, every label-preserving automorphism of $\overline\Gamma$ is the identity on every component $\Gamma_0$ of $\overline\Gamma$ for which there exists a closed path in $\Gamma_0$ whose label is non-trivial in $*_{i\in I}G_i$, then we say that $\Gamma$ satisfies the graphical $C_*(n)$-condition, respectively graphical $C'_*(\lambda)$-condition.
\end{defi}

A diagram over $\overline\Gamma$ is a diagram where every face $\Pi$ either bears the label of a simple closed path in $\overline\Gamma$ that is non-trivial in $*_{i\in I} G_i$, or $\Pi$ bears the label of a simple closed path in some $\Cay(G_i,S_i)$ and has no interior edge. The following analogy of Lemma~\ref{lem:graphical_basic} is stated in \cite[Lemma~3.8]{Gru-SQ}. A detailed proof is given in \cite[Theorem~1.35]{Gru-Thesis}. A $\overline\Gamma$-reduced diagram is a diagram over $\overline\Gamma$ in which no interior edge originates from $\overline\Gamma$ and in which every interior arc is locally geodesic.

\begin{lem}\label{lem:graphical_product_basic}
 Let $\Gamma$ be a $Gr_*(6)$-labelled graph over $S=\sqcup_{i\in I} S_i$, where $S_i$ are generating sets of groups $G_i$, and let $w$ be a word in $M(S)$. Then $w$ represents the identity in $G(\Gamma)_*$ if and only if there exists a $\overline\Gamma$-reduced diagram for $w$.
\end{lem}

We call a diagram $D$ as in Lemma~\ref{lem:graphical_product_basic} $\overline\Gamma$-reduced. We record the following generalization of \cite[Theorem 1]{Ste}, which follows from the arguments of \cite[Theorem 1]{Ste}, \cite[Theorem 2.16]{Gru}, and \cite[Lemma~4.3]{Gru} and refer the reader to \cite[Theorem~2.9]{Gru-Thesis} for a detailed proof. We first recall the definition  of a relatively hyperbolic group following \cite{Os-rh}.

\begin{defi}\label{defi:relative_presentation} Let $G$ be a group and $\{G_i\mid i\in I\}$ a collection of subgroups. Denote by $R_i$ all elements of $M(G_i)$ that represent the identity in $G_i$. A \emph{presentation of $G$ relative to $\{G_i\mid i\in I\}$} is a pair of sets $(X,R)$ such that $R\subseteq M(\sqcup_{i\in I} G_i\sqcup X)$, and $\langle \sqcup_{i\in I}G_i\sqcup X\mid \sqcup_{i\in I}R_i\sqcup R\rangle$ is a presentation of $G$ that is compatible with the inclusion maps $G_i\to G$. The \emph{relative area} of a word $w\in M(\sqcup_{i\in I}G_i\sqcup X)$ that represents the identity in $G$, denoted $\Area(w)$, is the minimal number of faces with labels in $R$ in a diagram over $\langle \sqcup_{i\in I}G_i\sqcup X\mid \sqcup_{i\in I}{R_i}\sqcup R\rangle$ whose boundary word is $w$. The \emph{relative Dehn function} associated to $(X,R)$ is the map $\N\to\N, n\mapsto\sup\{\Area(w)\mid w\in M(\sqcup_{i\in I} G_i\sqcup X), w=1\in G, |w|\leqslant n\}$
\end{defi}

\begin{defi} A group $G$ is \emph{hyperbolic relative to a collection of subgroups $\{G_i\mid i\in I\}$} if it admits a presentation $(X,R)$ relative to $\{G_i\mid i\in I\}$ such that $X$ and $R$ are finite and the associated relative Dehn function is bounded from above by a linear map. A group $G$ is \emph{non-trivially relatively hyperbolic} if it is hyperbolic relative to a collection of proper subgroups.
\end{defi}

\begin{thm}[{\cite[Theorem~2.9]{Gru-Thesis}}]\label{thm:relatively_hyperbolic} Let $\Gamma$ be a $Gr_*(7)$-labelled graph. Let $\overline R$ be the set of all words read on simple closed paths in $\overline\Gamma$, and let $R$ be a set words such that, for each $r\in \overline R$, there exist a cyclic shift $r'$ of $r$ and $r''\in R$ such that $r'$ and $r''$ represent the same element of $*_{i\in I}G_i$. Then $(\emptyset,R)$ is a presentation of $G(\Gamma)_*$ relative to the collection of subgroups $\{G_i\mid i\in I\}$ with a linear relative Dehn function. If $\Gamma$ is finite, then $R$ may be chosen to be finite and, hence, if $\Gamma$ is finite, then $G(\Gamma)_*$ is hyperbolic relative to $\{G_i\mid i\in I\}$.
\end{thm}

\subsection{Facts about small cancellation diagrams}

In our proofs, we will use properties of the following particular type of $(3,7)$-diagrams. If $\Pi$ is a face, then $i(\Pi)$ denotes the number of interior maximal arcs in $\partial\Pi$, and $e(\Pi)$ denotes the number exterior maximal arcs in $\partial \Pi$.

\begin{defi}\label{defi:n-gons}
 A \emph{$(3,7)$-$n$-gon} is a $(3,7)$-diagram with a decomposition of $\partial D$ into $n$ reduced subpaths $\partial D=\gamma_1\gamma_2\dots\gamma_n$ with the following property: Every face $\Pi$ of $D$ with $e(\Pi)=1$ for which the maximal exterior arc in $\partial\Pi$ is contained in one of the $\gamma_i$ satisfies $i(\Pi)\geqslant 4$. 
A face $\Pi$ for which there exists an exterior arc in $\partial\Pi$ that is not contained in any $\gamma_i$ is called \emph{distinguished}.
We use the words \emph{bigon}, \emph{triangle} and \emph{quadrangle} for $2$-gon, $3$-gon and $4$-gon. 
\end{defi}

\begin{thm}[Strebel's bigons, {\cite[Theorem 35]{Str}}]\label{thm:strebel_bigons} Let $D$ be a simple disk diagram that is a $(3,7)$-bigon. Then $D$ is either a single face, or it has shape $\mathrm{I}_1$ as depicted in Figure~\ref{figure:bigons}. Having shape $\mathrm{I}_1$ means:
\begin{itemize}
\item There exist exactly 2 distinguished faces. For each distinguished face $\Pi$, there exist an interior maximal arc $\delta_1$, and an exterior maximal arc $\delta_2$ such that $\partial\Pi=\delta_1\delta_2$.
\item For every non-distinguished face $\Pi$, there exist exterior maximal arcs $\delta_1$ and $\delta_3$ that are subpaths of the two sides of $D$ and that are not both subpaths of the same side, and interior maximal arcs $\delta_2$ and $\delta_4$ such that $\delta=\delta_1\delta_2\delta_3\delta_4$.
\end{itemize}
\end{thm}

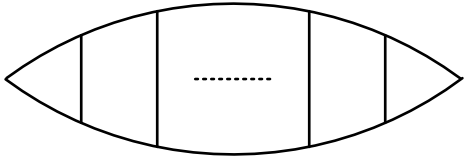
\begin{figure}\label{figure:bigons}
\begin{center}
\begin{tikzpicture}[line cap=round,line join=round,x=1.0cm,y=1.0cm,line width=1pt]
\draw [shift={(2,-3)}] plot[domain=0.93:2.21,variable=\t]({1*5*cos(\t r)+0*5*sin(\t r)},{0*5*cos(\t r)+1*5*sin(\t r)});
\draw [shift={(2,5)}] plot[domain=4.07:5.36,variable=\t]({1*5*cos(\t r)+0*5*sin(\t r)},{0*5*cos(\t r)+1*5*sin(\t r)});
\draw (0,1.58)-- (0,0.42);
\draw (4,1.58)-- (4,0.42);
\draw (1,1.9)-- (1,0.1);
\draw (3,1.9)-- (3,0.1);
\draw [dotted] (1.5,1)-- (2.5,1);
\end{tikzpicture}
\end{center}
\vspace{-12pt}
\caption{A diagram $D$ of \emph{shape $\mathrm{I}_1$}. All faces except the two distinguished ones are optional, i.e. $D$ may have as few as 2 faces. 
}
\end{figure}

We also record two useful formulas:

\begin{lem}[{\cite[Corollary~3.3]{LS}}]\label{lem:curvature_lyndon} Let $D$ be a diagram with at least 2 vertices, such that every interior vertex has degree at least 3 and in which every face has boundary length at least 6. Then
$$\sum_{v\in \partial D} 2+\frac{1}{2} -d(v)\geqslant 3,$$
where $v$ denotes vertices of $D$.
\end{lem}

\begin{lem}[{\cite[p.241]{Str}}]\label{lem:curvature_strebel} Let $D$ be a diagram without vertices of degree 2 such that every edge is contained in a face. Then
$$6=2\sum_{v}(3-d(v))+\sum_{e(\Pi)=k}(6-2k-i(\Pi)),$$
where $v$ denotes vertices and $\Pi$ denotes faces of $D$.
\end{lem}

%\begin{remark}\label{remark:geodesic_triangles}
%Strebel also provided a classification of $(3,7)$-triangles \cite[Theorem 43]{Str}. While his results are originally stated for reduced van Kampen diagrams over (not necessarily finite) classical $C'(\frac{1}{6})$-presentation whose boundaries decompose into two, respectively three, geodesic words, the proofs actually apply to our notion of $(3,7)$-bigon, respectively $(3,7)$-triangle. 
%Using Lemma~\ref{lem:graphical_basic}, it is an easy observation that if $D$ is a minimal diagram over a $Gr'(\frac{1}{6})$-labelled graph $\Gamma$ for a word $w$, where $w=w_1w_2$, respectively $w=w_1w_2w_3$, and each $w_i$ labels a geodesic in $\Cay(G(\Gamma),S)$, then $D$ is a $(3,7)$-bigon, respectively $(3,7)$-triangle. Therefore, we deduce that Strebel's results also apply to geodesic bigons and triangles in the Cayley graphs of graphical $Gr'(\frac{1}{6})$-groups.
%\end{remark}

\subsection{Embedding $\Gamma$ into $\Cay(G(\Gamma))$}

Given a component $\Gamma_0$ of an $S$-labelled graph $\Gamma$, after choosing a base vertex in $\Gamma_0$ and its image $\Cay(G(\Gamma),S)$, the labelling induces a unique label-preserving graph homomorphism $\Gamma_0\to\Cay(G(\Gamma),S)$. We show that, assuming $\Gamma$ has a $Gr'(\frac{1}{6})$-labelling, the image of every component of $\Gamma$ is convex in $\Cay(G(\Gamma),S)$. We also prove that it is isometrically embedded. The isometric embedding result was proved in \cite{Oll} for finite graphs assuming a stronger condition than our graphical $C'(\frac{1}{6})$-condition. It was observed in \cite{Gru} that Ollivier's isometric embedding result extends to arbitrary $Gr'(\frac{1}{6})$-labelled graphs.

\begin{lem}\label{lem:convex_embedding}
 Let $\Gamma_0$ be a component of a $Gr'(\frac{1}{6})$-labelled graph $\Gamma$, and let $f$ be a label-preserving graph homomorphism $\Gamma_0\to\Cay(G(\Gamma),S)$. Then $f$ is an isometric embedding, and its image is convex.
\end{lem}

\begin{proof}
Denote $X:=\Cay(G(\Gamma),S)$. Let $p_\Gamma$ be a geodesic path in $\Gamma_0$, and let $q_X$ be a geodesic path in $X$ that has the same endpoints as $f(p_\Gamma)$. 
Let $D$ be a $\Gamma$-reduced diagram for $\ell(p_\Gamma)\ell(q_X)^{-1}$ over $\Gamma$, and denote $\partial D=pq^{-1}$, i.e. $p$ is a lift of $p_\Gamma$ and $q$ is a lift of $q_X$. If $D$ has no faces, then $q_X=f(p_\Gamma)$, and the claim holds. From now on assume that $D$ contains at least one face. Note that, by the $Gr'(\frac{1}{6})$-assumption, for any face $\Pi$ of $D$, any interior arc $a$ in $\partial\Pi$ satisfies $|a|<\frac{|\partial\Pi|}{6}$.
 
Let $\Pi$ be a face of $D$. Since $q_X$ is geodesic, any arc $a$ in $\partial \Pi\cap q$ satisfies $|a|\leqslant\frac{|\partial\Pi|}{2}$. Suppose there exists an arc $a$ in $\partial \Pi\cap p$, and suppose a lift of $a$ via $\partial\Pi$ equals the lift of $a$ via $p\mapsto p_\Gamma$. Then the lift of $a$ in $\Gamma$ is a geodesic subpath of a simple closed path $\gamma$ in $\Gamma$, where $|\gamma|=|\partial\Pi|$; therefore, $|a|\leqslant\frac{|\partial\Pi|}{2}$. 
If the lifts are distinct for every choice of lift of $\partial\Pi$, then $a$ is a piece, and, hence, $|a|<\frac{|\partial\Pi|}{6}$. Therefore, $D$ is a $(3,7)$-bigon, and every disk component is either a single face, or it has shape $\mathrm{I}_1$ as in Theorem~\ref{thm:strebel_bigons}. 
 
Let $\Pi$ be a face of $D$. Then $\Pi$ has interior degree at most two, and any interior arc in $\partial\Pi$ is shorter than $\frac{|\partial \Pi|}{6}$. Since $|\partial\Pi\cap q|\leqslant\frac{|\partial\Pi|}{2}$, we obtain $|\partial\Pi\cap p|>\frac{|\partial\Pi|}{6}$. Therefore, the path $a:=\partial\Pi\cap p$ is not a piece, and a lift of $a$ to $\Gamma$ via $\partial \Pi$ and the lift of $a$ via $p\mapsto p_\Gamma$ are equal. 
Since this holds for every face $\Pi$, there exists a label-preserving graph homomorphism of the 1-skeleton of $D$ to $\Gamma_0$ that induces the lift $p\mapsto p_\Gamma$. This implies that $q_X$ lies in $f(\Gamma_0)$, whence the image of $f(\Gamma_0)$ is convex. Since $q_X$ lifts to a path in $\Gamma_0$ with the same endpoints as $p_\Gamma$, and since $p_\Gamma$ is geodesic, we have $|q_X|\geqslant|p_\Gamma|$. Thus, the map $\Gamma_0\to X$ is an isometric embedding.
\end{proof}

%By Theorem~\ref{thm:graphical_product_basic}, minimal diagrams over graphical $Gr_*(\frac{1}{6})$-labelled graphs have the same geometry as minimal diagrams over graphical $Gr(\frac{1}{6})$-labelled graphs. Therefore, the same proof, replacing $\Gamma$ by $\overline\Gamma$, shows: %Therefore, the above proof shows that for a graphical $Gr_*(\frac{1}{6})$-labelled graph $\Gamma$, we have that $\overline\Gamma$ is isometrically embedded and convex in $G(\Gamma)_*$. We state a particular consequence:

\begin{remark}\label{remark:convex_product} By Lemma~\ref{lem:graphical_product_basic}, the proof and statement of Lemma~\ref{lem:convex_embedding} also apply to the free product case replacing $\Gamma$ with $\overline\Gamma$, i.e.\ if $\Gamma$ is a $Gr'_*(\frac{1}{6})$-labelled graph over $*_{i\in I}G_i$ with generating sets $(S_i)_{i\in I}$, then each component of a $\overline\Gamma$ isometrically embeds into $\Cay(G(\Gamma)_*,\sqcup_{i\in I}S_i)$ and has a convex image.
In particular, if a component of $\Gamma$ embeds isometrically into $\overline\Gamma$, then it embeds isometrically into $\Cay(G(\Gamma)_*,\sqcup_{i\in I}S_i)$. Moreover, if an attached $\Cay(G_{i_0},S_{i_0})$ embeds isometrically into $\overline\Gamma$, then $\Cay(G_{i_0},S_{i_0})$ embeds isometrically into $\Cay(G(\Gamma)_*,\sqcup_{i\in I}S_i)$. Thus, if $I$ and all $S_i$ are finite, then $G_{i_0}$ embeds quasi-isometrically into $G(\Gamma)_*$ (where both groups are considered with their corresponding word-metrics). 

In order for an attached $\Cay(G_{i_0},S_{i_0})$ to be isometrically embedded (and convex) in $\overline\Gamma$, it is sufficient that the label-preserving automorphism group of $\overline\Gamma$ does not act transitively on the union of all vertex-sets of all attached $\Cay(G_{i_0},S_{i_0})$: If it does not act transitively, 
then every geodesic path in $\Cay(G_{i_0},S_{i_0})$ is a piece. The small cancellation condition ensures that any geodesic path in $\Cay(G_{i_0},S_{i_0})$ that is a piece is a geodesic path in $\overline\Gamma$, and any other geodesic path with the same endpoints is contained in the same copy of $\Cay(G_{i_0},S_{i_0})$.
\end{remark}

We also show that, assuming the weaker $Gr(6)$-condition the intersection of any two embedded components of $\Gamma$ is either empty or connected, which again carries over to the situation over free products.

\begin{lem}\label{lem:connected_embedding}
Let $\Gamma_1$ and $\Gamma_2$ be components of a $Gr(6)$-labelled graph $\Gamma$, and for each $i\in\{1,2\}$, let $f_i$ be a label-preserving graph-homomorphism $\Gamma_i\to\Cay(G(\Gamma),S)$. Then $f_1(\Gamma_1)\cap f_2(\Gamma_2)$ is either empty or connected.
\end{lem}

\begin{proof}
Let $x$ and $y$ be vertices in $f_1(\Gamma_1)\cap f_2(\Gamma_2)$. Denote $X:=\Cay(G(\Gamma),S)$, and let $p_X$, respectively $q_X$, be paths in $\Cay(G(\Gamma),S)$ from $x$ to $y$ such that $p_X=f_1(p_\Gamma)$ for a path $p_\Gamma$ in $\Gamma_1$ and $q_X=f_2(q_\Gamma)$ for a path $q_\Gamma$ in $\Gamma_2$. Assume that, given $x$ and $y$, $p_X$ and $q_X$ are chosen such that there exists a $\Gamma$-reduced diagram $D$ for $\ell(p_X)\ell(q_X)^{-1}$ over $\Gamma$ whose number of edges is minimal among all possible choices for $p_X$ and $q_X$. 
Denote $\partial D=pq^{-1}$, i.e.\ $p$ lifts to $p_X$ and $q$ lifts to $q_X$. Note that by our minimality assumptions, the only (possible) vertices of $D$ having degree 1 are the initial or terminal vertices of $p$ (or equivalently $q$).

For every face $\Pi$, any arc $a$ in $\partial\Pi\cap p$ or in $\partial\Pi\cap q$ is a piece since, otherwise, we could remove edges in $a$ as in Figure~\ref{figure:minimality}. Moreover, every interior arc is a piece since $D$ is $\Gamma$-reduced. Now iteratively remove all vertices of degree 2, except the initial and terminal vertices of $p$ (in case they have degree 2), by always replacing the two adjacent edges by a single one. 
This yields a $[3,6]$-diagram $\Delta$, where at most two vertices have degree less than 3. Thus, by Lemma~\ref{lem:curvature_lyndon}, $\Delta$ is either a single vertex or a single edge. This implies $p=q$, whence $p_X=q_X$ and, therefore, $p_X=q_X$ is a path in $f_1(\Gamma_1)\cap f_2(\Gamma_2)$ from $x$ to $y$.
\end{proof}

\begin{figure}\label{figure:minimality}
\hspace{2cm}
\begin{tikzpicture}[line width=1pt,>=stealth]
 \draw[-] (-.5,0) to (2.5,0);
% \draw[-,line width=3pt] (0,0) to (2,0);
 \draw[-,rounded corners=5] (0,0) to (0,-2) to (2,-2) to (2,0);
 \draw[->, dotted] (2.5,.1) to (-.5,.1);
 \draw [->,rounded corners=5,dashed,shorten <=2.5pt] (.1,-1.1) to (.1,-1.9) to (1.9,-1.9) to (1.9,-.9);
 \draw [->,rounded corners=5,dashed,shorten <=2.5pt] (1.9,-.9) to (1.9,-.1) to (.1,-.1) to (.1,-1.1); 
 \node at (1,-1) {\small $\Pi$};
\end{tikzpicture}
\hfill
\begin{tikzpicture}[line width=1pt,>=stealth]
 \draw[-] (-.5,0) to (0,0); \draw[-] (2,0) to (2.5,0);
% \draw[-,line width=3pt] (0,0) to (2,0);
 \draw[-,rounded corners=5] (0,0) to (0,-2) to (2,-2) to (2,0);
 \draw[->, dotted, rounded corners=5] (2.5,.1) to (1.9,.1) to (1.9,-1.9) to (.1,-1.9) to (.1,.1) to (-.5,.1);
\end{tikzpicture}
\hspace{2cm}
\caption{
Left: The dotted line represents a subpath $p$ of $\partial D$ that lifts to a path $p_\Gamma$ in $\Gamma$, the dashed line represents the boundary cycle $\partial \Pi$ of a face $\Pi$. Right:  We remove edges of a path $a$ in $p\cap\partial\Pi$ and, thus, remove $\Pi$. If the lift of $a$ in $\Gamma$ via $p$ essentially equals a lift of $a$ via $\partial\Pi$, then the resulting path, drawn as dotted line, lifts to a path in $\Gamma$ with the same endpoints as $p_\Gamma$. If the two lifts are essentially distinct, then $a$ is a piece.
}
\end{figure}
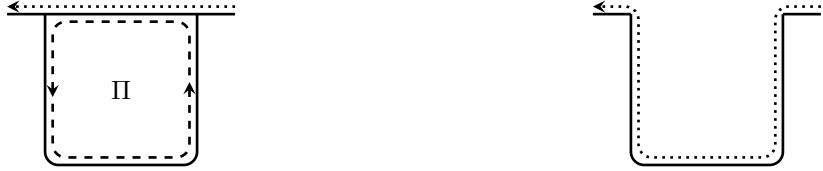

\section{The hyperbolic space}\label{section:hyperbolic_space}

In the first part of this section, we construct, given a $Gr(7)$-labelled graph $\Gamma$, a Gromov hyperbolic Cayley graph of $G(\Gamma)$. More generally, we construct for a group $G$ defined by a (possibly infinite) presentation $\langle S\mid R\rangle$ satisfying a certain subquadratic isoperimetric inequality a (possibly non-locally finite) Gromov hyperbolic Cayley graph $Y$ of $G$. In the second part, we provide a description of the geodesics in the hyperbolic space in the case that $\Gamma$ is $Gr'(\frac{1}{6})$-labelled and use it to show that the action of $G(\Gamma)$ on $Y$ is not acylindrical in general.

\subsection{Construction of the space} We will prove:

\begin{thm}\label{thm:hyperbolic_space} Let $\Gamma$ be a $Gr(7)$-labelled graph over a set $S$, and let $W$ be the set of all elements of $G(\Gamma)$ represented by labels of paths in $\Gamma$. Then $\Cay(G(\Gamma),S\cup W)$ is hyperbolic.
\end{thm}

Our argument rests on the following proposition, which will be deduced from Theorem~\ref{thm:subquadratic}.

\begin{prop}\label{prop:hyperbolic_space} Let $\langle S\mid R\rangle$ be a presentation of a group $G$,  %where $S\subseteq G$ and 
where $R\subseteq M(S)$ is closed under cyclic conjugation and inversion. Let $W_0$ be the set of all subwords of elements of $R$. Suppose there exists a subquadratic map $f:\N\to\N$ with the following property for every $w\in M(S)$: If $w$ is trivial in $G$ and if $w$ can be written as product of $N$ elements of $W_0$, then there exists a diagram for $w$ over $\langle S\mid R\rangle$ with at most $f(N)$ faces. Denote by $W$ the image of $W_0$ in $G$. Then $\Cay(G,S\cup W)$ is Gromov hyperbolic. 
\end{prop}

We can think of the space $\Cay(G,S\cup W)$ as obtained from the Cayley 2-complex of $\langle S\mid R\rangle$ by replacing the every 2-cell by the complete graph on its vertices. Our proof uses the following result of Bowditch, and our argument also applies in the more general context of simply-connected 2-complexes.

\begin{thm}[{\cite{Bowditch-subquadratic}}]\label{thm:subquadratic}
 Let $Y$ be a connected graph, let $\Omega$ be the set of all closed paths in $Y$, and let $A:\Omega\to \N$ be a map satisfying:
 \begin{itemize}
  \item If $\gamma_1,\gamma_2,\gamma_3$ are closed paths with the same initial vertex and if $\gamma_3$ is homotopic to $\gamma_1\gamma_2$, then $A(\gamma_3)\leqslant A(\gamma_1)+A(\gamma_2)$.
  \item If $\gamma\in\Omega$ is split into four subpaths $\gamma=\alpha_1\alpha_2\alpha_3\alpha_4$, then $A(\gamma)\geqslant d_1d_2$, where $d_1=d(\alpha_1,\alpha_3)$ and $d_2=d(\alpha_2,\alpha_4)$.
 \end{itemize}
Here $d$ is the graph-metric. If $\sup\{A(\gamma)\mid \gamma\in\Omega,|\gamma|\leqslant n\}=o(n^2)$, then $Y$ is Gromov hyperbolic.
\end{thm}

\begin{proof}[Proof of Proposition~\ref{prop:hyperbolic_space}]
Let $Y:=\Cay(G,S\cup W)$. If $w\in W_0\subseteq M(S)$, we denote by $\dot w$ the image of $w$ in $W\subseteq G$. Consider the presentation $\langle S\cup W\mid R\cup R_W\rangle$, where $R_W:=\{w\dot w^{-1}\mid w\in W_0\}\subseteq M(S\cup W)$. This is a presentation of $G$. 

Let $\gamma$ be a closed path in $Y$. Then the label $\ell(\gamma)$ of $\gamma$ admits a diagram $D$ over $\langle S\cup W\mid R\cup R_W\rangle$ such that $D$ has at most $|\gamma|$ boundary faces and such that every interior edge of $D$ is labelled by an element of $S$, i.e. all interior faces have labels in $R$.
If $D$ has a minimal number of faces among all diagrams for $\ell(\gamma)$, then, by construction, $D$ has at most $|\gamma|+f(|\gamma|)$ faces. 
For a closed path $\gamma$, denote by $A(\gamma)$ the minimal number of faces of a diagram for $\ell(\gamma)$ over $\langle S\cup W\mid R\cup R_W\rangle$. Then $\sup\{A(\gamma)\mid \gamma\in\Omega,|\gamma|\leqslant n\}$ is a subquadratic map as required. The map $A$ moreover satisfies the first assumption of Theorem~\ref{thm:subquadratic}.

To prove the second assumption of Theorem~\ref{thm:subquadratic}, it is sufficient to consider the case that $\gamma$ is a simple closed path, as the general case can be constructed from this. Let $\gamma$ be decomposed into four subpaths $\gamma=\alpha_1\alpha_2\alpha_3\alpha_4$, and let $d_1:=d(\alpha_1,\alpha_3)$ and $d_2:=(\alpha_2,\alpha_4)$. We may assume that $d_1>0$ and $d_2>0$. Let $D$ be a simple disk diagram for the label of $\gamma$ with a minimal number of faces. 
By definition of $Y$, any two vertices in the image the 1-skeleton of a face of $D$ in $Y$ are at distance at most 1 from each other. Thus, no path in $D$ connecting $\alpha_1$ to $\alpha_3$ (respectively connecting $\alpha_2$ to $\alpha_4$) is contained in strictly fewer than $d_1$ (respectively $d_2$) faces. Induction on $d_1$ (or $d_2)$ yields that $D$ has at least $d_1d_2$ faces, i.e. $A(\gamma)\geqslant d_1d_2$. Therefore, we can apply Theorem~\ref{thm:subquadratic}.
\end{proof}

\begin{proof}[Proof of Theorem~\ref{thm:hyperbolic_space}]
 This follows from Proposition~\ref{prop:hyperbolic_space} by considering the presentation $\langle S\mid R\rangle$ of $G(\Gamma)$, where $R$ is the set of all labels of closed paths in $\Gamma$. Let $W_0$ be the set of all labels of paths in $\Gamma$, and let $w=w_1\dots w_N$ for $w_i\in W_0$ such that $w$ is trivial in $G(\Gamma)$. Then there exists a diagram for $w$ over $\langle S\mid R\rangle$ with at most $N$ boundary faces. Let $D$ be a diagram with a minimal number of edges among all such diagrams. Then the arguments of \cite[Lemma~2.10]{Gru} yield that $D$ has no interior edge originating from $\Gamma$ and that every interior face has a freely non-trivial boundary word. Therefore, $D$ is a $(3,7)$-diagram and, thus, has at most $8N$ faces by \cite[Proposition~2.7]{Str}.
\end{proof}

In Remark~\ref{remark:hyperbolic_space}, we provide alternative arguments showing that $\Cay(G(\Gamma),S\cup W)$ is Gromov hyperbolic which do not rely on Proposition~\ref{prop:hyperbolic_space} but on geometric features specific to $(3,7)$-bigons and $(3,7)$-triangles.

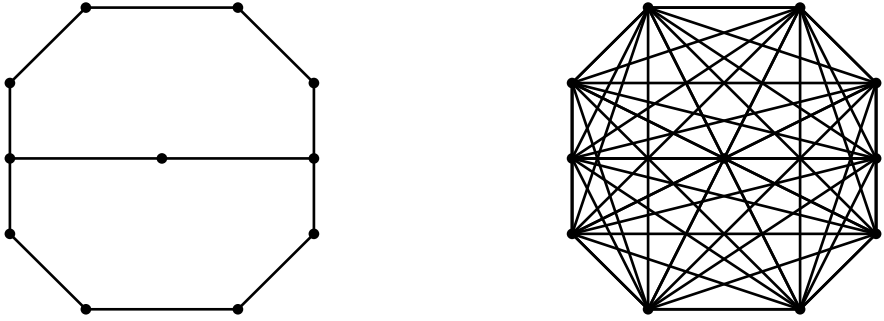
\begin{figure}\label{figure:hyperbolic_space}
\begin{center}
\begin{tikzpicture}[line width=1pt]
%\node at (-2,0) {\small $\Gamma_i$ in $X$:};

\node[coordinate] at (0,1) (X1) {};
\node[coordinate] at (1,2) (X2) {};
\node[coordinate] at (3,2) (X3) {};
\node[coordinate] at (4,1) (X4) {};

\node[coordinate] at (0,0) (Y1) {};
\node[coordinate] at (2,0) (Y2) {};
\node[coordinate] at (4,0) (Y3) {};

\node[coordinate] at (0,-1) (Z1) {};
\node[coordinate] at (1,-2) (Z2) {};
\node[coordinate] at (3,-2) (Z3) {};
\node[coordinate] at (4,-1) (Z4) {};

\fill (X1) circle (2pt);\fill (X2) circle (2pt);\fill (X3) circle
(2pt);\fill (X4) circle (2pt);
\fill (Y1) circle (2pt);\fill (Y2) circle (2pt);\fill (Y3) circle
(2pt);%\fill (Y4) circle (2pt);
\fill (Z1) circle (2pt);\fill (Z2) circle (2pt);\fill (Z3) circle
(2pt);\fill (Z4) circle (2pt);

\draw[-] (X1) -- (X2);
\draw[-] (X2) -- (X3);
\draw[-] (X3) -- (X4);

\draw[-] (Y1) -- (X1);\draw[-] (X4) -- (Y3);
\draw[-] (Y1) -- (Y2);\draw[-] (Y2) -- (Y3);
\draw[-] (Y1) -- (Z1);\draw[-] (Z4) -- (Y3);

\draw[-] (Z1) -- (Z2);
\draw[-] (Z2) -- (Z3);
\draw[-] (Z3) -- (Z4);
\end{tikzpicture}
\hspace{3cm}
\begin{tikzpicture}[line width=1pt]
%\node at (-2,0) {\small $\Gamma_i$ in $Y$:};

\node[coordinate] at (0,1) (X1) {};
\node[coordinate] at (1,2) (X2) {};
\node[coordinate] at (3,2) (X3) {};
\node[coordinate] at (4,1) (X4) {};

\node[coordinate] at (0,0) (Y1) {};
\node[coordinate] at (2,0) (Y2) {};
\node[coordinate] at (4,0) (Y3) {};

\node[coordinate] at (0,-1) (Z1) {};
\node[coordinate] at (1,-2) (Z2) {};
\node[coordinate] at (3,-2) (Z3) {};
\node[coordinate] at (4,-1) (Z4) {};

%\draw[-,color=red] (X1) -- (X1);
\draw[-] (X1) -- (X2);
\draw[-] (X1) -- (X3);
\draw[-] (X1) -- (X4);
\draw[-] (X1) -- (Y1);
\draw[-] (X1) -- (Y2);
\draw[-] (X1) -- (Y3);
\draw[-] (X1) -- (Z1);
\draw[-] (X1) -- (Z2);
\draw[-] (X1) -- (Z3);
\draw[-] (X1) -- (Z4);

%\draw[-] (X2) -- (X1);
%\draw[-] (X2) -- (X2);
\draw[-] (X2) -- (X3);
\draw[-] (X2) -- (X4);
\draw[-] (X2) -- (Y1);
\draw[-] (X2) -- (Y2);
\draw[-] (X2) -- (Y3);
\draw[-] (X2) -- (Z1);
\draw[-] (X2) -- (Z2);
\draw[-] (X2) -- (Z3);
\draw[-] (X2) -- (Z4);

%\draw[-] (X3) -- (X1);
%\draw[-] (X3) -- (X2);
%\draw[-] (X3) -- (X3);
\draw[-] (X3) -- (X4);
\draw[-] (X3) -- (Y1);
\draw[-] (X3) -- (Y2);
\draw[-] (X3) -- (Y3);
\draw[-] (X3) -- (Z1);
\draw[-] (X3) -- (Z2);
\draw[-] (X3) -- (Z3);
\draw[-] (X3) -- (Z4);

%\draw[-] (X4) -- (X1);
%\draw[-] (X4) -- (X2);
%\draw[-] (X4) -- (X3);
%\draw[-] (X4) -- (X4);
\draw[-] (X4) -- (Y1);
\draw[-] (X4) -- (Y2);
\draw[-] (X4) -- (Y3);
\draw[-] (X4) -- (Z1);
\draw[-] (X4) -- (Z2);
\draw[-] (X4) -- (Z3);
\draw[-] (X4) -- (Z4);

%\draw[-] (Y1) -- (X1);
%\draw[-] (Y1) -- (X2);
%\draw[-] (Y1) -- (X3);
%\draw[-] (Y1) -- (X4);
%\draw[-] (Y1) -- (Y1);
\draw[-] (Y1) -- (Y2);
\draw[-] (Y1) -- (Y3);
\draw[-] (Y1) -- (Z1);
\draw[-] (Y1) -- (Z2);
\draw[-] (Y1) -- (Z3);
\draw[-] (Y1) -- (Z4);

%\draw[-] (Y2) -- (X1);
%\draw[-] (Y2) -- (X2);
%\draw[-] (Y2) -- (X3);
%\draw[-] (Y2) -- (X4);
%\draw[-] (Y2) -- (Y1);
%\draw[-] (Y2) -- (Y2);
\draw[-] (Y2) -- (Y3);
\draw[-] (Y2) -- (Z1);
\draw[-] (Y2) -- (Z2);
\draw[-] (Y2) -- (Z3);
\draw[-] (Y2) -- (Z4);

%\draw[-] (Y3) -- (X1);
%\draw[-] (Y3) -- (X2);
%\draw[-] (Y3) -- (X3);
%\draw[-] (Y3) -- (X4);
%\draw[-] (Y3) -- (Y1);
%\draw[-] (Y3) -- (Y2);
%\draw[-] (Y3) -- (Y3);
\draw[-] (Y3) -- (Z1);
\draw[-] (Y3) -- (Z2);
\draw[-] (Y3) -- (Z3);
\draw[-] (Y3) -- (Z4);

%\draw[-] (Z1) -- (X1);
%\draw[-] (Z1) -- (X2);
%\draw[-] (Z1) -- (X3);
%\draw[-] (Z1) -- (X4);
%\draw[-] (Z1) -- (Y1);
%\draw[-] (Z1) -- (Y2);
%\draw[-] (Z1) -- (Y3);
%\draw[-] (Z1) -- (Z1);
\draw[-] (Z1) -- (Z2);
\draw[-] (Z1) -- (Z3);
\draw[-] (Z1) -- (Z4);

%\draw[-] (Z2) -- (X1);
%\draw[-] (Z2) -- (X2);
%\draw[-] (Z2) -- (X3);
%\draw[-] (Z2) -- (X4);
%\draw[-] (Z2) -- (Y1);
%\draw[-] (Z2) -- (Y2);
%\draw[-] (Z2) -- (Y3);
%\draw[-] (Z2) -- (Z1);
%\draw[-] (Z2) -- (Z2);
\draw[-] (Z2) -- (Z3);
\draw[-] (Z2) -- (Z4);

%\draw[-] (Z3) -- (X1);
%\draw[-] (Z3) -- (X2);
%\draw[-] (Z3) -- (X3);
%\draw[-] (Z3) -- (X4);
%\draw[-] (Z3) -- (Y1);
%\draw[-] (Z3) -- (Y2);
%\draw[-] (Z3) -- (Y3);
%\draw[-] (Z3) -- (Z1);
%\draw[-] (Z3) -- (Z2);
%\draw[-] (Z3) -- (Z3);
\draw[-] (Z3) -- (Z4);

\fill (X1) circle (2pt);\fill (X2) circle (2pt);\fill (X3) circle
(2pt);\fill (X4) circle (2pt);
\fill (Y1) circle (2pt);\fill (Y2) circle (2pt);\fill (Y3) circle
(2pt);%\fill (Y4) circle (2pt);
\fill (Z1) circle (2pt);\fill (Z2) circle (2pt);\fill (Z3) circle
(2pt);\fill (Z4) circle (2pt);

\draw[-] (X1) -- (X2);
\draw[-] (X2) -- (X3);
\draw[-] (X3) -- (X4);

\draw[-] (Y1) -- (X1);\draw[-] (X4) -- (Y3);
\draw[-] (Y1) -- (Y2);\draw[-] (Y2) -- (Y3);
\draw[-] (Y1) -- (Z1);\draw[-] (Z4) -- (Y3);

\draw[-] (Z1) -- (Z2);
\draw[-] (Z2) -- (Z3);
\draw[-] (Z3) -- (Z4);

\end{tikzpicture}
\end{center}
\vspace{-12pt}
\caption{Left: The graph induced by the image of a component of the labelled graph $\Gamma$ in $\Cay(G(\Gamma),S)$. Right: The (complete) graph induced by the image of a component of $\Gamma$ in $Y=\Cay(G(\Gamma),S\cup W)$.}
\end{figure}

The arguments of Theorem~\ref{thm:hyperbolic_space}, replacing $\Gamma$ with $\overline\Gamma$, also yield the following:

\begin{thm}\label{thm:relative_hyperbolic_space} Let $\Gamma$ be a $Gr_*(7)$-labelled graph over a free product $\freeproduct_{i\in I}G_i$, and let $W$ be the set of all elements of $G(\Gamma)_*$ represented by labels of paths in $\overline\Gamma$. Then $\Cay(G(\Gamma)_*,\sqcup_{i\in I}G_i\cup W)$ is hyperbolic.
\end{thm}

\begin{remark}\label{remark:relative_hyperbolic_space} Theorem~\ref{thm:relative_hyperbolic_space} can be considered as an application of Proposition~\ref{prop:hyperbolic_space} to a relative presentation having a subquadratic relative Dehn function. Let $R$ be the set all words read on closed paths in $\Gamma$ (not $\overline\Gamma$). Then, by Theorem~\ref{thm:relatively_hyperbolic}, $(\emptyset, R)$ is a presentation of $G(\Gamma)_*$ relative to $\{G_i\mid i\in I\}$ with a linear relative Dehn function. Let $W'$ be the set of all elements of $G(\Gamma)_*$ represented by subwords of elements of $R$. Then $\Cay(G(\Gamma)_*,\sqcup_{i\in I}G_i\cup W')$ is quasi-isometric to $\Cay(G(\Gamma)_*,\sqcup_{i\in I}G_i\cup W)$ as in Theorem~\ref{thm:relative_hyperbolic_space} and, hence, hyperbolic. 
Therefore, $G(\Gamma)_*$ is weakly hyperbolic relative to $W'$ and $\{G_i\mid i\in I\}$ in the sense of \cite[Definition 4.1]{DGO}. 
\end{remark}

\subsection{Geodesics in the $Gr'(\frac{1}{6})$-case}

We show that geodesics in $\Cay(G(\Gamma),S\cup W)$ are close to geodesics in  $\Cay(G(\Gamma),S)$ in the case that $\Gamma$ is $Gr'(\frac{1}{6})$-labelled by providing a description of the geodesics in $\Cay(G(\Gamma),S)$. Applying our construction, we show that the action of $G(\Gamma)$ on $\Cay(G(\Gamma),S\cup W)$ is not acylindrical in general, even in the case of classical $C'(\frac{1}{6})$-groups.

\begin{prop}\label{prop:geodesics} Let $\Gamma$ be a $Gr'(\frac{1}{6})$-labelled graph, and let $W$ be the set of all elements of $G(\Gamma)$ represented by words read on $\Gamma$. Let $x\neq y$ be vertices in $X:=\Cay(G(\Gamma),S)$ and $\gamma_X$ a geodesic in $X$ from $x$ to $y$. Denote $k:=d_Y(x,y)$, where $Y:=\Cay(G(\Gamma),S\cup W)$. Then:
\begin{itemize}
 \item $k$ is the minimal number such that $\gamma_X=\gamma_{1,X}\dots\gamma_{k,X}$, where each $\gamma_{i,X}$ is a lift of a path in $\Gamma$ or (the inverse of) an edge labelled by an element of $S$ that does not occur on $\Gamma$. This means $(\iota \gamma_{1,X},\tau\gamma_{1,X})\dots(\iota\gamma_{k,X},\tau\gamma_{k,X})$ is a geodesic in $Y$ from $x$ to $y$.
 \item If $\Gamma_1,\dots,\Gamma_k$ are images of components of $\Gamma$ or of single edges labelled by elements of $S$ that do not occur on $\Gamma$ such that $\Gamma_1\cup\dots\cup\Gamma_k\subseteq X$ contains a path from $x$ to $y$, then $\gamma_X$ is contained in $\Gamma_1\cup\dots\cup\Gamma_k$ and intersects each $\Gamma_i$ in at least one edge. 
\end{itemize}
\end{prop}

Here, if $v\neq w$ are vertices, then $(v,w)$ denotes an edge $e$ with $\iota e=v$ and $\tau e=w$.

\begin{remark}
 Let $x\neq y$ be vertices in $\Cay(G(\Gamma),S)$. The sequence $\Gamma_1,\Gamma_2,\dots,\Gamma_k$ considered in Proposition~\ref{prop:geodesics} has the following properties:
 \begin{itemize}
  \item $\bigcup_{i=1}^k\Gamma_i$ contains every geodesic in $\Cay(G(\Gamma),S)$ from $x$ to $y$.
  \item $\bigcup_{i=1}^k\Gamma_i$ is connected and no three $\Gamma_i$ pairwise intersect. There exist (not necessarily distinct) $\Gamma_{i_0}$ and $\Gamma_{i_1}$ that each intersect at most one other $\Gamma_i$. \end{itemize}
The second part follows from the minimality of $k$. $\Gamma_{i_0}$ and $\Gamma_{i_1}$ are the components containing $x$ and $y$, respectively.

This aspect of our result generalizes a part of \cite[Theorem 4.15]{ADr}, which was the first description of geodesics for infinitely presented small cancellation groups: in \cite[Theorem 4.15]{ADr}, given any classical small $C'(\frac{1}{8})$-group $G$, for any two vertices $x\neq y\in \Cay(G,S)$, a sequence $\Gamma_1,\Gamma_2,\dots,\Gamma_k$ with the two above properties is constructed. In this case, each $\Gamma_i$ is either an embedded cycle graph labelled by a relator or a single edge. \cite[Theorem 4.15]{ADr} is not concerned with a minimal number $k$ of components and provides further details on metric properties of the sequence such as pairwise distance of non-intersecting $\Gamma_i$.
\end{remark}

\begin{proof}[Proof of Proposition~\ref{prop:geodesics}]
 For the proof, assume that every letter occurs on an edge of $\Gamma$. If this is a priori not the case, it can be achieved by adding for each $s\in S$ that does not occur on $\Gamma$ a new component to $\Gamma$ that is simply an edge labelled by $s$.
 
 Let $x\neq y$ be vertices in $X$, and let $k:=d_Y(x,y)$. Let $\gamma_X$ be a geodesic in $X$ from $x$ to $y$, and let $l$ be minimal such that $\gamma_X=\gamma_{1,X}\gamma_{2,X}\dots\gamma_{l,X}$, where each $\gamma_{i,X}$ lifts to a path in $\Gamma$. We will show that $l=k$.
 
 Since $d_Y(x,y)=k$, there exists a path $\sigma_X$ in $X$ from $x$ to $y$ such that there exist images $\Gamma_i$ in $\Cay(G(\Gamma),S)$ of components of $\Gamma$ such that $\sigma_X=\sigma_{1,X}\sigma_{2,X}\dots\sigma_{k,X}$ where each $\sigma_{i,X}$ is a path in $\Gamma_i$. This gives rise to a lift $\sigma_{i,\Gamma}$ in $\Gamma$ of each $\sigma_{i,X}$. 
 We choose $\sigma_X$ and $\sigma_{i,X}$ as above such that $|\sigma_X|$ is minimal and such that, for every $j<k$, $\sum_{r=1}^j|\sigma_{r,X}|$ is maximal. Note that, since $|\sigma_X|$ is minimal, $\sigma_X$ is labelled by a reduced word, and each $\sigma_{i,X}$ lifts to a geodesic in $\Gamma$.
 
 Let $D$ be a $\Gamma$-reduced diagram for $\ell(\sigma_X)\ell(\gamma_X)^{-1}$ over $\Gamma$, i.e.\ we can write $\partial D=\sigma\gamma^{-1}$ where $\sigma$ lifts to $\sigma_X$ and $\gamma$ lifts to $\gamma_X$. Denote by $\sigma_i$ the lifts in $D$ of the $\sigma_{i,X}$. 
 
 \vspace{12pt}
 
 \noindent {\bf Claim 1.} $D$ is a $(3,7)$-bigon. 
 
 Let $\Pi$ be a face of $D$ with $e(\Pi)=1$. Then there exists a unique maximal exterior arc $p$ in $\partial\Pi$. If $p$ is contained in $\gamma$, then $|p|\leqslant\frac{|\partial\Pi|}{2}$ since $\gamma_X$ is a geodesic, whence $i(\Pi)\geqslant 4$. 
 
 Now suppose $p$ is contained in $\sigma$, and suppose that $i(\Pi)\leqslant 3$. Then $|p|>\frac{|\partial\Pi|}{2}$, whence $p$ is not a concatenation of at most 3 pieces. Since $d(x,y)=k$, the concatenation of two consecutive $\sigma_i$ cannot lift to a path in $\Gamma$. Therefore, $p$ must be a subpath of $\sigma_{i_0}\sigma_{i_0+1}\sigma_{i_0+2}$ for some $i_0$.  
 Since $p$ is not a concatenation of at most 3 pieces, there exists $j\in\{i_0,i_0+1,i_0+2\}$ for which $p\cap\sigma_j$ is not a piece. Therefore, a lift of $p\cap\sigma_j$ via $\partial \Pi$ equals the lift via $\sigma_{j}\mapsto\sigma_{j,\Gamma}$. This implies that in the decomposition $\sigma_{1,X}\sigma_{2,X}\dots \sigma_{k,X}$, 
 we can replace $\sigma_{j,X}$ by a lift $\tilde\sigma_{j,X}$ of $p$ such that $\tilde\sigma_{j,X}$ is a path in $\Gamma_j$, and we correspondingly shorten the paths $\sigma_{j-1,X}$ and $\sigma_{j+1,X}$ that are (possibly) intersected by $\tilde\sigma_{j,X}$. 
 (By minimality of $k$, no other paths are intersected.) The resulting decomposition $\sigma_X=\tilde\sigma_{1,X}\tilde\sigma_{2,X}\dots\tilde\sigma_{k,X}$ still satisfies that every $\tilde\sigma_{i,X}$ is path in $\Gamma_{i}$. 
 Since $\tilde\sigma_{j,X}$ is contained in a lift in $\Gamma_j$ of $\partial\Pi$ with $|\tilde\sigma_{j,X}|=|p|>\frac{|\partial\Pi|}{2}$, we have that $\tilde\sigma_{j,X}$ is not a geodesic path in $\Gamma_j$. Thus, we can replace $\tilde\sigma_{j,X}$ by a shorter path in $\Gamma_j$, contradicting the minimality of $|\sigma_X|$. Therefore, $i(\Pi)\geqslant 4$.
 
 \vspace{12pt}
 
\noindent {\bf Claim 2.} $D^{(1)}$ maps to $\Gamma_1\cup \Gamma_2\cup\dots\cup\Gamma_k$. Since, each $\Gamma_i$ is convex by Lemma~\ref{lem:convex_embedding}, this proves that $l=k$, and our proposition follows. 
 
Suppose $D$ contains a disk component $\Delta$. Then we can number the faces of $\Delta$ by $\Pi_1,\Pi_2,...$ starting from the one closest to $\iota \sigma$. (This makes sense since $\Delta$ is a single face or has shape $\mathrm{I_1}$ by Theorem~\ref{thm:strebel_bigons}.) Consider $\Pi_1$. Denote by $\sigma_0$ a path of length $0$. As argued above, the path $\partial\Pi_1\cap\sigma$ is contained in $\sigma_{i-1}\sigma_{i}\sigma_{i+1}$ for some $i\geqslant 1$. 
By maximality of $\sum_{r=1}^{i}|\sigma_{r,X}|$, it is a subpath of $\sigma_{i-1}\sigma_{i}$ for some $i$ such that $\partial\Pi_1\cap \sigma_i$ contains an edge. We have $|\partial\Pi_1\cap \gamma|\leqslant\frac{|\partial\Pi_1|}{2}$ since $\gamma_X$ is a geodesic. 
Moreover, $\partial\Pi_1$ has at most one subpath $p$ that is a maximal interior arc, and $|p|<\frac{|\partial\Pi_1|}{6}$. Therefore, $|\partial\Pi_1\cap \sigma|>\frac{|\partial\Pi_1|}{3}$, whence $\partial\Pi_1\cap \sigma$ cannot be a concatenation of two pieces. 
Thus, by maximality of $\sum_{r=1}^{i-1}|\sigma_{r,X}|$, the lift of $\partial\Pi_1\cap \sigma_i$ via $\sigma_i\mapsto \sigma_{i,\Gamma}$ must equal a lift via $\partial\Pi_1$. 
 
Now suppose $\Pi_2$ exists. Then $\partial\Pi_2\cap \sigma$ is a subpath of $\sigma_i\sigma_{i+1}$ ($i$ from the above paragraph), and $\partial\Pi_2\cap \sigma$ has an initial subpath $\sigma_i'$ that is a (possibly empty) terminal subpath of $\sigma_i$. By the above observation, the lifts of $\partial\Pi_1$ give rise to lifts of the concatenation $q:=(\partial \Pi_1\cap\partial\Pi_2) \sigma_i'$ in $\Gamma$. 
Note that $q$ is contained in $\partial\Pi_2$ and, thus, has lifts via $\partial\Pi_2$ in $\Gamma$. Since no interior edge of $D$ originates from $\Gamma$, these lifts are never equal, whence $q$ is a piece. Therefore, the same argument as above shows that the lift of $\partial\Pi_2\cap \sigma_{i+1}$ via $\sigma_{i+1}\mapsto\sigma_{i+1,\Gamma}$ equals a lift via $\partial\Pi_2$. Claim 2 follows inductively.
\end{proof}

\begin{remark}  In the case of a $Gr_*'(\frac{1}{6})$-labelled graph over a free product, the above proof and, hence, result apply if $\Gamma$ is replaced by $\overline\Gamma$. The only additional observation required is that any geodesic in $X$ that lifts to $\overline\Gamma$ is locally geodesic.\end{remark}

Proposition~\ref{prop:geodesics} lets us study the action of $G(\Gamma)$ on $Y$. We use it to show that the action need not be acylindrical in general. 

\begin{defi}[{\cite[Introduction]{Os-acyl}}]\label{defi:acylindrical} A group $G$ acts \emph{acylindrically} on a metric space $Y$ if for every $\epsilon >0$ there exist $K\in\N$ and $N\in\N$ such that for every $x,y\in Y$ with $d(x,y)\geqslant K$, there exist at most $N$ elements $g\in G$ satisfying:
$$d(x,gx)\leqslant \epsilon \text{ and } d(y,gy)\leqslant \epsilon.$$
\end{defi}

\begin{example}\label{example:not_acylindrical}
We construct a classical $C'(\frac{1}{6})$-presentation $\langle S\mid R\rangle$ of a group $G$ such that the action of $G$ on $Y:=\Cay(G,S\cup W)$ is not acylindrical. Here $W$ is the set of all elements of $G$ represented by subwords of elements of $R$. 
This corresponds to our above definition of $W$ by taking $\Gamma$ to be the disjoint union of cycle graphs labelled by the elements of $R$.

Let $G$ be defined by a classical $C'(\frac{1}{6})$-presentation $\langle S\mid R\rangle$ with the following property for every $N\in \N$: There exists a cyclically reduced word $w_N\in M(S)$ satisfying the following conditions. (Denote by $\omega_N$ a 1-infinite ray in $\Cay(G,S)$ starting at $1\in G$ with label $w_Nw_N\dots$.)
\begin{itemize}
\item[a)] $w_N^N$ is a subword of a relator in $R$.
\item[b)] If $\gamma$ is a path in $\Cay(G,S)$ with label in $R$ and if $p$ is a path in $\gamma\cap\omega_N$, then $|p|\leqslant\frac{|\gamma|}{6}$.
\item[c)] There exists an integer $C_N$ such that if $\gamma$ is a path in $\Cay(G,S)$ with label in $R$ and if $p$ is a path in $\gamma\cap\omega_N$, then $|p|\leqslant C_N|w_N|$.
\end{itemize}

Let $N\in\N$. By Theorem~\ref{thm:strebel_bigons}, b) implies that every subpath of $\omega_N$ is a geodesic in $\Cay(G,S)$. Therefore, Proposition~\ref{prop:geodesics} and c) yield for every $K\in\N$ and $L:=C_NK$ that $d_Y(1,w_N^{L})\geqslant K.$ By a), for every $0\leqslant m\leqslant N$ we have
$$d_Y(1,w_N^m)=d_Y(w_N^{L},w_N^mw_N^{L})\leqslant 1.$$
The elements of $G$ represented by $w_N,w_N^2,\dots,w_N^N$ are pairwise distinct. We conclude that, for every $K\in\N$ and every $N\in\N$, there exist points $x$ and $y$ in $Y$ and at least $N$ elements $g$ of $G$ satisfying:
$$d_Y(x,y)\geqslant K\text{ and }d_Y(x,gx)=d_Y(y,gy)\leqslant 1.$$
Therefore, the action of $G$ on $Y$ is not acylindrical.

The (symmetrized closure) of the presentation $\langle a,b,s_1,s_2,\dots s_{12}\mid r_1,r_2,\dots\rangle$ with $$r_N:=(ab^N)^Ns_1^{N^2+N}s_2^{N^2+N}\dots s_{12}^{N^2+N}$$ is a classical $C'(\frac{1}{6})$-presentation that satisfies the above conditions with $w_N=ab^N$ and $C_N=N$.
\end{example}

\section{The WPD element}\label{section:wpd_element}

In this section, we complete the proofs of Theorems~\ref{thm:gr7} and \ref{thm:c7} by showing the existence of a WPD element for the action of $G(\Gamma)$ on the hyperbolic space constructed in Theorem~\ref{thm:hyperbolic_space}. We then provide a slight refinement for the case of $Gr'(\frac{1}{6})$-groups, and we show that all results hold for the corresponding free product small cancellation cases as well.

\subsection{The graphical $Gr(7)$ and $C(7)$-cases}
From now until the end of this section, we fix a $Gr(7)$-labelled graph for the proof of Theorem~\ref{thm:gr7}, respectively a $C(7)$-labelled graph $\Gamma$ for the proof of Theorem~\ref{thm:c7}, with a set of labels $S$ such that the following hold:
\begin{itemize}\label{graph_assumptions}
 \item Every $s\in S$ occurs on an edge of $\Gamma$.
 \item No $s\in S$ occurs on exactly one edge of $\Gamma$.
 \item $\Gamma$ has at least one component, and every component of $\Gamma$ has a non-trivial fundamental group.
 \item No two components $\Gamma_1$ and $\Gamma_2$ of $\Gamma$ admit a label-preserving isomorphism $\Gamma_1\to\Gamma_2$.
\item In the case of Theorem~\ref{thm:gr7}, $\Gamma$ has at least two finite components $\Gamma_1,\Gamma_2$.
\item In the case of Theorem~\ref{thm:c7}, $\Gamma$ contains at least one embedded cycle graph $c$. 
\end{itemize}

We explain why these properties can be assumed for the proofs: If the first property does not hold for $s\in S$, then $s$ generates a free factor in $G(\Gamma)$, and either $G(\Gamma)$ is isomorphic to $\Z$ or to $G'*\Z$ for some non-trivial group $G'$. In both cases, the statements of the theorems hold. 

For the second property, let $e$ be an edge whose label $s$ occurs on no other edge of $\Gamma$. 
The operation of removing $e$ from $\Gamma$ and simultaneously removing $s$ from the alphabet corresponds to a Tietze-transformation if $e$ is contained in an embedded cycle graph. If $e$ is not contained in an embedded cycle graph, then the operation corresponds to projecting to the identity the free factor of $G(\Gamma)$ that is the infinite cyclic group generated by $s$. Thus, if we simultaneously remove all such edges and the corresponding labels from the alphabet, the resulting graph defines either $G(\Gamma)$, or it defines a group $G'$ such that $G(\Gamma)\cong G'*F$ for some non-trivial free group 
$F$. In the latter case, the statements of the theorems hold. 

The third and fourth properties can be arranged by simply discarding superfluous components. If no component remains, $G(\Gamma)$ is a free group. 

For the last two properties, if, in either case, the property is not satisfied, then $\Gamma$ is finite or a forest, and $G(\Gamma)$ is Gromov hyperbolic (if it is finitely generated) by \cite[Theorem~2.16]{Gru} or a non-trivial free product. Therefore, the statements of the theorems hold.

\begin{lem}\label{lem:c7_wpd1}
Let $\Gamma_0$ be a finite component of $\Gamma$. Then one of the following holds:
 \begin{itemize}
  \item There exist distinct vertices $x$ and $y$ in $\Gamma_0$ such that no path from $x$ to $y$ is a concatenation of pieces.
  \item There exists a simple closed path in $\Gamma_0$ that is a concatenation of pieces.
 \end{itemize}
\end{lem}

\begin{proof} 
 Suppose the first claim does not hold. Then any two vertices of $\Gamma_0$ can be connected by a path that is a concatenation of pieces. If every edge in $\Gamma_0$ is a piece, the second claim holds since $\Gamma_0$ has a non-trivial fundamental group. Now assume an edge $e$ is not a piece. 
 
 Since the label of $e$ occurs more than once on $\Gamma$ and since no two components of $\Gamma$ are isomorphic, there exists a label-preserving automorphism $\phi:\Gamma_0\to\Gamma_0$ with $\phi(e)\neq e$. Let $p$ be a reduced path from $\iota e$ to $\phi(\iota e)$ that is a concatenation of pieces. Since $p$ is not closed, its label is freely non-trivial. 
 Since $\Gamma_0$ is finite, there exists $k>0$ with $\phi^k=\operatorname{id}$, whence the path $p\phi(p)\dots\phi^{k-1}(p)$ is closed. It is non-trivial since its label is freely non-trivial. Therefore, its reduction contains a simple closed path that is a concatenation of pieces.
\end{proof}

\begin{lem}\label{lem:gr7_wpd}
Suppose $\Gamma$ has distinct finite components $\Gamma_1$ and $\Gamma_2$. 
Then there exist vertices  $x_1,y_1$ and $x_2,y_2$ in $\Gamma$ with the following properties:% for all paths $\overline\alpha_1:x_1\to y_1,\overline\alpha_2:x_2\to y_2$:
 \begin{itemize}
  \item $x_1$ and $y_1$ lie in the same component of $\Gamma$, and $x_2$ and $y_2$ lie in the same component of $\Gamma$. Moreover, $x_1\neq y_1$, $x_2\neq y_2$, $x_2$ and $y_1$ are essentially distinct, and $y_2$ and $x_1$ are essentially distinct.
 \item If $\overline\alpha_1=p^{-1}\alpha_1q$ is a path from $x_1$ to $y_1$, where $p$ is a piece and $q$ is a lift of a path terminating at $x_2$, then $\alpha_1$ is not a piece.
 \item If $\overline\alpha_2=q^{-1}\alpha_2p$ is a path from $x_2$ to $y_2$, where $q$ is a lift of a path terminating at $y_1$ and $p$ is a piece, then $\alpha_2$ is not a piece.
  \item There exists at most one reduced path $\overline\alpha_1=p^{-1}\alpha_1q$ from $x_1$ to $y_1$ such that $p$ is a lift of a path terminating at $y_2$, $q$ is a lift of a path terminating at $x_2$, and $\alpha_1$ is a concatenation of at most two pieces.
 \item There exists at most one reduced path $\overline\alpha_2=q^{-1}\alpha_2p$ from $x_2$ to $y_2$ such that $q$ is a lift of a path terminating at $y_1$, $p$ is a lift of a path terminating at $x_1$, and $\alpha_2$ is a concatenation of at most two pieces.
 \end{itemize}
\end{lem}

\begin{proof} 
Denote $X:=\Cay(G(\Gamma),S)$. First assume both $\Gamma_1$ and $\Gamma_2$ satisfy the second claim of Lemma~\ref{lem:c7_wpd1}.
 For $i=1,2$, let $\gamma_i$ be simple closed paths in $\Gamma_i$ that are concatenations of pieces, and denote their initial vertices by $v_i$. Consider the maps of labelled graphs $f_i:\Gamma_i\to X$ obtained by mapping $v_i$ to $1\in G(\Gamma)$. Let $C:=f_1(\Gamma_1)\cap f_2(\Gamma_2)$. The maps $f_i$ are injective by Lemma~\cite[Lemma~4.1]{Gru}, and $C$ is connected by Lemma~\ref{lem:connected_embedding}. Since $v_1$ and $v_2$ are essentially distinct, we have that any path in $C$ is a piece.

For each $i\in\{1,2\}$, let $p_i$ be the maximal subpath of $\gamma_i$ contained in $f_i^{-1}(X\setminus C)$. (See Figure~\ref{figure:free_gr7} for an illustration.) Then $p_i$ is not a concatenation of at most 5 pieces. Let $w_i$ be the initial vertex of the maximal terminal subpath of $p_i$ that is a concatenation of at most 3 pieces. 
Then $f_i(w_i)$ cannot be connected to any vertex of $C$ by any path in $f_i(\Gamma_i)$ that is a concatenation of at most two pieces, for else there would exist a non-trivial closed path in $\Gamma$ that is a concatenation of at most 6 pieces. Denote $x_1=w_1,y_1=v_1,x_2=v_2,y_2=w_2$. Then first claim holds since $\Gamma_1$ and $\Gamma_2$ are non-isomorphic, and our above observation proves the second and third claims.
 
If one or both $\Gamma_i$ satisfy the first claim of Lemma~\ref{lem:c7_wpd1}, then we make the above construction letting $v_i$ and $w_i$ be any distinct vertices in $\Gamma_i$ that cannot be connected by a path that is a concatenation of pieces.

 For the fourth claim, suppose there are two distinct reduced paths $\overline \alpha_1$ and $\overline\alpha_1'$ as in the claim. We write $\overline\alpha_1=p^{-1}\alpha_1q$ and $\overline\alpha_1'=p'^{-1}\alpha_1'q'$ as above. Note that each one of $pp'^{-1}$ and $q'q^{-1}$ is a piece by construction. 
 Therefore, $pp'^{-1}\alpha_1'q'q^{-1}\alpha_1^{-1}$ is a closed path that is a concatenation of at most 6 pieces, and it is non-trivial since the label of its cyclic conjugate $\overline\alpha_1'\overline\alpha_1^{-1}$ is freely non-trivial; this is a contradiction. For the fifth claim, the same argument applies.
\end{proof}

 \begin{figure}\label{figure:free_gr7}
\begin{center}
\begin{tikzpicture}[>=stealth,%shorten <=2.5pt, shorten >=2.5pt,
line width=1pt,x=2cm,y=2cm]

\draw[-,shift={(1,0)}] (0:1) -- (51.4:1) -- (102.9:1) -- (154.3:1) -- (205.7:1) -- (257.1:1) -- (308.6:1) -- (0:1);

\draw[-,shift={(-2.8,0)},rotate around={180:(1,0)}] (0:1) -- (51.4:1) -- (102.9:1) -- (154.3:1) -- (205.7:1) -- (257.1:1) -- (308.6:1) -- (0:1);
\node at (0,0) {$C$};

%\draw[-,shift={(1,0)}] (205.7:1) .. controls (257.1:1) and (308.6:1) .. (0:1);

%\draw (1,0) arc(-25.7:25.7:1);

\draw[->,dashed,shift={(1,0)},domain=25.7:334.3] plot ({-.8*cos(\x)},{.8*sin(\x)});
\draw[->,dashed,shift={(-.8,0)},domain=25.7:334.3] plot ({.8*cos(\x)},{.8*sin(\x)});

\node at (2.4,0) {\small $f_1(w_1)$};
\fill (2,0) circle (2pt);

\node at (-2.2,0) {\small $f_2(w_2)$};
\fill (-1.8,0) circle (2pt);

\fill (.1,-.43) circle (2pt);
\node at (.1,-.63) {\small $1$};
\end{tikzpicture}
\end{center}
\vspace{-12pt}
\caption{An illustration of the union of $f_1(\Gamma_1)$ (right) and $f_2(\Gamma_2)$ (left) in $X$. The intersection $f_1(\Gamma_1)\cap f_2(\Gamma_2)$ is denoted by $C$.
The dashed lines represent the paths $f_1(p_1)$ (right) and $f_2(p_2)$ (left). Note that $1=f_1(v_1)=f_2(v_2)$.
}
\end{figure}
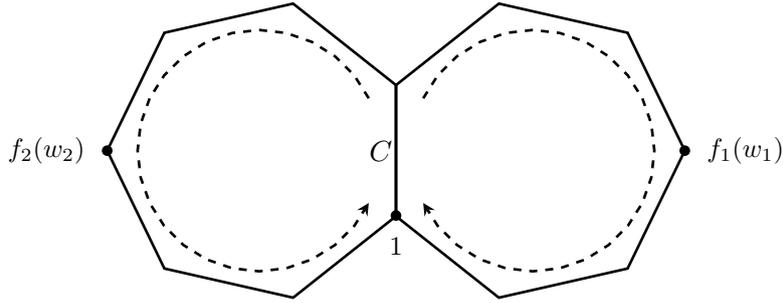

\begin{lem}\label{lem:c7_wpd} Suppose $\Gamma$ admits no non-trivial label-preserving automorphism. Let $c$ be an embedded cycle graph in $\Gamma$ such that every edge of $c$ is a piece. Then there exist vertices $x_1,y_1$ and $x_2,y_2$ in $c$ for which the statement of Lemma~\ref{lem:gr7_wpd} holds.
\end{lem}

\begin{proof}
Let $\gamma_1$ be a simple closed path based at a vertex $v_1$ whose image is $c$. Let $v_2$ be the terminal vertex of the longest initial subpath of $\gamma_1$ made up of at most 3 pieces, and let $\gamma_2$ be the cyclic shift of $\gamma_1$ with initial vertex $v_2$. (See Figure~\ref{figure:free_c7}.) Note that $v_1\neq v_2$ and, hence, $v_1$ and $v_2$ are essentially distinct because $\Gamma$ has no non-trivial label-preserving automorphism. 
We make the same construction as in Lemma~\ref{lem:gr7_wpd} for the component $\Gamma_0$ of $\Gamma$ containing $c$, i.e.\ we map $\Gamma_0$ to $\Cay(G(\Gamma),S)$ by $f_1(v_1)=1$ and by $f_2(v_2)=1$ and choose the vertices $w_1$ and $w_2$ as above. By construction, there is a path from $v_1$ to $w_2$ that is a concatenation of at most two pieces, whence $w_1\neq w_2$, and $w_1$ and $w_2$ are essentially distinct. All other claims of Lemma~\ref{lem:gr7_wpd} follow with the same proofs.
\end{proof}
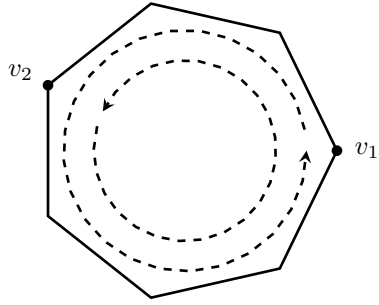
\begin{figure}\label{figure:free_c7}
\begin{center}
\begin{tikzpicture}[>=stealth,%shorten <=2.5pt, shorten >=2.5pt,
line width=1pt,x=2cm,y=2cm]

\draw[-] (0:1) -- (51.4:1) -- (102.9:1) -- (154.3:1) -- (205.7:1) -- (257.1:1) -- (308.6:1) -- (0:1);

\fill (0:1) circle (2pt);
\node at (0:1.2) {\small $v_1$};
\fill (154.3:1) circle (2pt);
\node at (154.3:1.2) {\small $v_2$};

\draw[->,dashed,domain=10:360] plot ({.8*cos(\x)},{.8*sin(\x)});
\draw[->,dashed,domain=164.3:514.3] plot ({.6*cos(\x)},{.6*sin(\x)});

\end{tikzpicture}
\end{center}
\vspace{-12pt}
\caption{An illustration of $c\subseteq \Gamma$. The outer dashed line represents $\gamma_1$, and the inner dashed line represents $\gamma_2$. 
%$\Gamma_1$ (right) and $\phi_2(\Gamma_2)$ (left) in $X$.  The dashed lines represent the paths $\phi_1(p_1)$ (right) and $\phi_2(p_2)$ (right). 
}
\end{figure}

Given vertices $v$ and $w$ in a labelled graph, we denote by $p:v\to w$ a path with $\iota p=v$ and $\tau p=w$, and we denote by $\ell(v\to w)$ the label of such a path. 

\medskip

\noindent {\bf Definition of the WPD element $g$.} In the notation of Lemmas~\ref{lem:gr7_wpd}, respectively \ref{lem:c7_wpd}, let $g$ be the element of $G(\Gamma)$ represented by $\ell(x_1\to y_1)\ell(x_2\to y_2)$.

\begin{remark}\label{remark:non-elementary}
It follows from \cite[Section 3]{Gru} that if, in the case of Theorem~\ref{thm:gr7}, $\Gamma$ has at least 4 pairwise non-isomorphic components or, in the case of Theorem~\ref{thm:c7}, $\Gamma$ contains two disjoint embedded cycle graphs, then $G(\Gamma)$ contains a free subgroup that is freely generated by two distinct elements $g$ as above. In particular, in these cases, $G(\Gamma)$ is not virtually cyclic.
\end{remark}

Denote $X:=\Cay(G(\Gamma),S)$ and $Y:=\Cay(G(\Gamma),S\cup W)$, where $W$ is the set of all elements of $G(\Gamma)$ represented by words read on $\Gamma$. If $\alpha_Y=(e_1,e_2,\dots,e_k)$ is a path in $Y$, then a path in $X$ 
\emph{representing} $\alpha_Y$ is a path $\alpha_X$ in $X$ together with a decomposition $\alpha_X=\alpha_{1,X}\alpha_{2,X}\dots\alpha_{k,X}$ with, for each $i$, $\iota e_i=\iota\alpha_{i,X}$ and $\tau e_i=\tau\alpha_{i,X}$ such that, for each $i$, a lift $\alpha_{i,\Gamma}$ in $\Gamma$ of $\alpha_{i,X}$ is chosen. 
We call the paths $\alpha_{i,X}$ \emph{segments}. Observe that if $\alpha_Y$ is a geodesic in $Y$ of length $k$, and if $\alpha_X$ is a path representing $\alpha_Y$, then any two vertices in $\alpha_X$ are at distance (in $Y$) at most $k$ from each other.

\begin{lem}\label{lem:c7_geodesic} Let $N\in\N$. Then there exists a path $\alpha_Y$ in $Y$ from $1$ to $g^N$ of length $2N$ with the following properties: 

\begin{itemize}
 \item There exists a reduced path $\alpha_X$ in $X$ representing $\alpha_Y$ and a decomposition $\alpha_X$ into segments $\alpha_{1,X},\alpha_{2,X},\dots,\alpha_{2N,X}$ with the following properties, where we denote by $\alpha_{i,\Gamma}$ the lift in $\Gamma$ of each $\alpha_{i,X}$.
\begin{itemize}
\item For every $i$, there exist paths $p_i$ in $X$ and $\overline\alpha_{i,\Gamma}:x_{\overline i}\to y_{\overline i}$ in $\Gamma$, where $\overline i\equiv i \mod 2$, such that $p_0$ and $p_{2N}$ have length 0 and such that for every $i$, the path $p_{i-1}^{-1}\alpha_{i,X}p_i$ lifts to $\overline\alpha_{i,\Gamma}$, and this lift induces the lift $\alpha_{i,X}\mapsto \alpha_{i,\Gamma}$.
\item Given $\alpha_Y$, for every choice of $\alpha_{1,X},\alpha_{2,X},\dots,\alpha_{2N,X}$ with the above properties, every $\alpha_{i,X}$ has length $>0$ and is not a piece.
\end{itemize} 
\item $\alpha_Y$ is a geodesic in $Y$.
\end{itemize}
\end{lem}

\begin{proof}
By definition of $g$, there exist paths $\overline\alpha_{i,X}$ in $X$ such that each $\overline\alpha_{2i-1,X}$ lifts to a path $\overline\alpha_{2i-1,\Gamma}:x_1\to y_1$, such that each $\overline\alpha_{2i,X}$ lifts to a path $\overline\alpha_{2i,\Gamma}:y_2\to x_2$, and such that $\overline\alpha_{1,X}\overline\alpha_{2,X}\dots\overline\alpha_{2N,X}$ is a path from $1$ to $g^N$ in $X$. 
The path $\alpha_X$ obtained as the reduction of this path satisfies the first part of the first statement and, conversely, any path satisfying the first part of the first statement can be constructed in this manner. The second part of the first statement now follows by definition of $g$, i.e.\ applying the assertions of Lemmas~\ref{lem:gr7_wpd}, respectively \ref{lem:c7_wpd}. 

We proceed to the proof of the second statement. Let $\beta_Y$ be a geodesic in $Y$ from $1$ to $g^N$ of length $k$. Choose paths $\alpha_X$ representing $\alpha_Y$ as above and $\beta_X$ representing $\beta_Y$ such that there exists a $\Gamma$-reduced diagram $D$ for $\ell(\alpha_X)\ell(\beta_X)^{-1}$ over $\Gamma$ whose number of edges is minimal among all possible choices. 
We denote $\partial D=\alpha\beta^{-1}$, i.e. $\alpha$ lifts to $\alpha_X$ and $\beta$ lifts to $\beta_X$. Note that if $e$ is an edge in $\alpha$, then the lift $\alpha\mapsto\alpha_X$ and the lifts of segments $\alpha_{i,X}\mapsto \alpha_{i,\Gamma}$ induce a lift of $e$ in $\Gamma$; the same observation holds for $\beta$.

\vspace{12pt}

\noindent {\bf Claim 1.} $D$ has no faces, whence $\alpha_X=\beta_X$.

Let $\Pi$ be a face, and let $e$ be an edge in $\partial\Pi\cap\alpha$. If a lift of $e$ via $\partial\Pi$ equals the lift via $\alpha_{i,X}\mapsto\alpha_{i,\Gamma}$ for some $i$, then we can remove $e$ from $D$ as in Figure~\ref{figure:minimality}, and we can remove any resulting spurs and fold together resulting consecutive edges with inverse labels to obtain a diagram with fewer edges than $D$ that satisfies our assumptions; a contradiction. 
The same observation holds for any edge in $\partial \Pi\cap\beta$. Therefore, any arc in the intersection of a face with the lift of a segment is a piece. 
 
No segment of $\alpha_X$ is a piece. Therefore, for any face $\Pi$, any path in $\partial\Pi\cap\alpha$ is a concatenation of at most 2 pieces. Suppose a path $p$ in $\partial\Pi\cap \beta$ lifts to a subpath of $\beta_X$ that is a concatenation of two segments. Then these two segments can be replaced by a single segment which corresponds to a path in the lift of $\partial \Pi$ in $\Gamma$, and $\beta_X$ can be decomposed into $k-1$ segments. This contradicts the fact that $\beta_Y$ is a geodesic. 
Therefore, any path in $\partial\Pi\cap\beta$ is a concatenation of at most 3 pieces. Thus, any face $\Pi$ with $e(\Pi)=1$ whose exterior edges are contained in $\alpha$ or in $\beta$ has interior degree at least 4. This implies that $D$ is a $(3,7)$-bigon, whence any of its disk components has shape $\mathrm{I_1}$ as in Theorem~\ref{thm:strebel_bigons}, or it has at most one face. 
 
If $D$ has at least one face, then there exist a face $\Pi$ such that $\partial\Pi$ is the concatenation of at most 3 arcs as follows: An arc $\gamma_1$ in $\alpha$, an arc $\gamma_2$ in $\beta$, and possibly an interior arc $\gamma_3$. By our above observation, this implies that $\gamma$ is a concatenation of no more than 6 pieces, a contradiction. Therefore, $D$ has no faces, whence $\alpha=\beta$ and $\alpha_X=\beta_X$.
 
\vspace{12pt}

\noindent {\bf Claim 2.} $k=2N$, whence $\alpha_Y$ is a geodesic.
 
We denote the decomposition into segments of $\beta_X$ as $\beta_X=\beta_{1,X}\beta_{2,X}\dots\beta_{k,X}$ and the lift in $\Gamma$ of $\beta_{i,X}$ by $\beta_{i,\Gamma}$. Since $k\leqslant 2N$, there exist $i$ and $j$ such that $\alpha_{i,X}$ is contained in $\beta_{j,X}$. Consider the lift $\alpha_{i,X}\mapsto\alpha_{i,\Gamma}$ and the lift of $\alpha_{i,X}$ via $\beta_{j,X}\mapsto \beta_{j,\Gamma}$. 
Since $\alpha_{i,X}$ is not a piece, these lifts are essentially equal. Therefore, the decomposition $\alpha=\alpha_{1,X}\alpha_{2,X}\dots\alpha_{i-1,X}'\beta_{j,X}\alpha_{i+1,X}'\dots \alpha_{2N,X}$ , where $\alpha_{i-1}'$ is an initial subpath of $\alpha_i$ and $\alpha_{i+1}'$ is a terminal subpath of $\alpha_{i+1}$, with the associated lifts (where the lift $\beta_{j,X}\mapsto\beta_{j,\Gamma}$ may have to be composed with an automorphism of $\Gamma$) is a decomposition as in the first statement; in particular no segment has length $0$ or is a piece.

We can now apply the above procedure to the initial subpath of $\alpha$ terminating at $\iota \beta_{j,X}$ and to the terminal subpath of $\alpha$ starting at $\tau\beta_{j,X}$. Induction yields that the decomposition $\alpha_X=\beta_{1,X}\beta_{2,X}\dots\beta_{k,X}$ is as in the first statement, whence $k=2N$. 
\end{proof}

\begin{cor} $g$ acts hyperbolically. 
\end{cor}

\begin{remark}\label{remark:hyperbolic_space} The arguments of claim 1 in the proof of Lemma~\ref{lem:c7_geodesic} show the following: 
Given two geodesics $\alpha_Y$ and $\beta_Y$ in $Y$ with the same endpoints, there exist paths $\alpha_X$ and $\beta_X$ in $X$ representing the $\alpha_Y$, respectively $\beta_Y$, such that there exist paths $\alpha_X'$ and $\beta_X'$ in $X$ with the same endpoints as $\alpha_X$ and $\beta_X$ such that (denoting by $d_H$ the Hausdorff-distance in $Y$, where $Y$ is considered as a geodesic metric space) $d_H(\alpha_X,\alpha_X')\leqslant 2$ and $d_H(\beta_X,\beta_X')\leqslant 2$,
and there exists a diagram $D$ with a boundary path $\alpha'\beta'^{-1}$, where $\alpha'$ is a lift of $\alpha_X'$ and $\beta'$ is a lift of $\beta_X'$, such that $D$ is a $(3,7)$-bigon. Hence, every disk component of $D$ has shape $\mathrm{I}_1$, whence $d_H(\alpha_X',\beta_X')\leqslant 2$.
This implies $d_H(\alpha_Y,\beta_Y)\leqslant 10$ and, thus, geodesic bigons in $Y$ are uniformly thin. Therefore, $Y$ is Gromov hyperbolic by \cite{Papasoglu-bigons} independently of Theorem~\ref{thm:subquadratic}.

Another way to prove Gromov hyperbolicity of $Y$ is observing, as above, that geodesic triangles in $Y$ are close to triangles in $X$ that give rise to $(3,7)$-triangles over $\Gamma$. Such triangles are $3$-slim by Strebel's classification of $(3,7)$-triangles \cite[Theorem 43]{Str}.
\end{remark}

\begin{prop}\label{prop:wpd7}
 $g$ satisfies the WPD condition.
\end{prop}

\begin{proof} Let $K>0$, and let $N_0$ such that $d_Y(1,g^N)>2K+5$ for all $N\geqslant N_0$. Let $N\geqslant N_0$, and let $h\in G(\Gamma)$ with $d_Y(1,h)\leqslant K$ and $d_Y(1,g^{-N}hg^{N})\leqslant K$. We will show that, given $K$ and $N_0$, there exist only finitely many possibilities for choosing $h$. Let $D$ be a $\Gamma$-reduced diagram with the following properties, where $\partial D=\alpha\delta_1\beta^{-1}\delta_2^{-1}$.

\begin{itemize}
 \item $\alpha$ lifts to a reduced path $\alpha_X$ in $X$ representing a geodesic $1\to g^N$ in $Y$ with a decomposition as in the statement of Lemma~\ref{lem:c7_geodesic}.
 \item $\beta$ lifts to a reduced path $\beta_X$ in $X$ representing a geodesic $1\to g^N$ in $Y$ with a decomposition as in the statement of Lemma~\ref{lem:c7_geodesic}.
 \item $\delta_1$ lifts to a path $\delta_{1,X}$ in $X$ representing a geodesic $1\to g^{-N}hg^N$ in $Y$.
 \item $\delta_2$ lifts to a path $\delta_{2,X}$ in $X$ representing a geodesic $1\to h$ in $Y$. 
 \item Among all such choices, the number of edges of $D$ is minimal.
\end{itemize}

Given $D$, we make additional minimality assumptions on the decompositions of $\alpha_X$ and $\beta_X$:
Denote the decompositions $\alpha_X=\alpha_{1,X}\alpha_{2,X}\dots\alpha_{2N,X}$ and $\beta_X=\beta_{1,X}\beta_{2,X}\dots\beta_{2N,X}$, and denote by $\alpha_i$, respectively $\beta_j$, the lifts of $\alpha_{i,X}$, respectively $\beta_{j,X}$ in $D$. Denote the lifts in $\Gamma$ of $\alpha_i$, respectively $\beta_j$, by $\alpha_{i,\Gamma}$, respectively $\beta_{j,\Gamma}$ and the corresponding paths $x_1\to y_1$ or $x_2\to y_2$ by $\overline\alpha_{i,\Gamma}$, 
respectively $\overline\beta_{j,\Gamma}$. We assume that, given $\alpha_X$ and $\beta_X$, the decompositions and their lifts are chosen such that both $\sum_{i=1}^{2N}|\overline\alpha_{i,\Gamma}|$ and $\sum_{j=1}^{2N}|\overline\beta_{j,\Gamma}|$ are minimal. 
Since $\alpha_X$ and $\beta_X$ are reduced, this readily implies that every $\overline\alpha_{i,\Gamma}$ and every $\overline\beta_{j,\Gamma}$ is a reduced path. Also, observe that our assumptions on $D$ imply that both $\delta_1$ and $\delta_2$ are reduced paths.

\vspace{12pt}

\noindent {\bf Claim 1.} $D$ has no faces.

By minimality, for any face $\Pi$ and any $i,j$, any path in $\partial\Pi\cap \alpha_i$ or $\partial\Pi\cap \beta_j$ is a piece since, otherwise, we could remove edges as in Figure~\ref{figure:minimality} and subsequently remove any resulting spurs and fold away any resulting consecutive inverse edges. 
The same observation holds for any path in $\partial \Pi\cap\delta$, where $\delta$ is a subpath of $\delta_1$ or $\delta_2$ that is a lift of a segment of $\delta_{1,X}$ or $\delta_{2,X}$.

No $\alpha_i$ or $\beta_j$ is a piece, whence for any face $\Pi$ we have that any path in $\partial\Pi\cap\alpha$  or in $\partial\Pi\cap\beta$ lies in the concatenation of no more than two $\alpha_i$, respectively $\beta_j$, and, thus, it is a concatenation of no more than two pieces. Suppose for a face $\Pi$, there exists a path $\delta$ in $\delta_1$ (or in $\delta_2$) that is a lift of a segment such that $\delta$ lies in $\partial\Pi$. 
Then we can remove the edges of $\delta$ from $D$, thus replacing $\delta$ by a path $\delta'$ such that $\partial\Pi=\delta\delta'^{-1}$. 
The resulting path $\delta_{1,X}'$ (or $\delta_{2,X}'$) can be decomposed with the same number of segments, contradicting the minimality assumptions on $D$. Therefore, any path in $\partial\Pi\cap\delta_1$ or in $\partial\Pi\cap\delta_2$ is a subpath of the concatenation of at most two lifts of segments and, therefore, a concatenation of at most two pieces.  This shows that $D$ is 
a $(3,7)$-quadrangle. 

Let $\Delta$ be a disk component of $D$. If there exist 4 distinguished faces, then every distinguished face of $\Delta$ with exterior degree 1 intersects at most two sides of $\Delta$ in arcs and thus has interior degree at least 3. 
This contradicts Lemma~\ref{lem:curvature_strebel} (after removing vertices of degree 2), since any such distinguished face contributes at most 1 positive curvature, and the only positive contributions come from distinguished faces with exterior degree 1. Similarly, the existence of 3 distinguished faces yields a contradiction. 

Thus, there exist at most two distinguished faces, whence $\Delta$ is a $(3,7)$-bigon and, by Theorem~\ref{thm:strebel_bigons}, it is of shape $\mathrm{I}_1$. Note that $\Delta$ must intersect all 4 sides of $D$: If $\Pi$ is a distinguished face of $\Delta$, then its boundary path cannot be made up of fewer than 7 pieces. 
Hence, since its interior degree is 1, $\Pi$ must intersect at least 3 sides because the intersection of $\Pi$ with any side is made up of at most 2 pieces. Considering shape $\mathrm{I}_1$, we also see that there cannot exist a non-distinguished face, since such a face would have a boundary path made up of at most 6 pieces. Thus $\Delta$ has at most two faces. 
The lifts $\delta_{1,X}$ and $\delta_{2,X}$ of $\delta_1$ and $\delta_2$ represent geodesics in $Y$ of length at most $K$, whence, for each $i$, any two vertices in $\delta_{i,X}$ are at $Y$-distance at most $K$ from each other. Any two vertices in the image in $Y$ of the 1-skeleton a face of $D$ at are at distance at most 1 from each other by definition of $Y$. 
Therefore, the assumption that $d_Y(1,g^N)>2K+5>2K+2$ implies that $\Delta$ cannot contain vertices of both $\delta_1$ and $\delta_2$, whence $\Delta$ does not exist. Thus, $D$ has no faces. 

\vspace{12pt}
\noindent {\bf Claim 2.} Given $K$ and $N_0$, there exist only finitely many possibilities for $h$.

Recall that $\alpha$ and $\beta$ lift to paths in $X$ representing geodesics in $Y$, and $\delta_2$ lifts to a path in $X$ representing a geodesic of length at most $K$ in $Y$. Therefore, $\delta_2$ is contained in $(\alpha_1\alpha_2\dots\alpha_{K+1})\cup(\beta_1\beta_2\dots\beta_{K+1})$. Each $\alpha_i$ and each $\beta_j$ lifts to a path in either the component of $\Gamma$ containing $x_1$ or in the component of $\Gamma$ containing $x_2$. 
Therefore, if the components of $\Gamma$ containing $x_1$ and $x_2$ are both finite, there exist only finitely many possibilities for $h$. This completes the proof in the case of Theorem~\ref{thm:gr7}.

We proceed to show that it is actually sufficient for the components to have finite automorphism groups, which also completes the proof in the case of Theorem~\ref{thm:c7}, as in that case, the automorphism groups are trivial. Denote by $p$ a maximal path in $\alpha\cap\beta$. Applying our above observation on $\delta_2$ to $\delta_1$ and using the fact that $d_Y(1,g^N)>2K+5$ yields that there exist $i_0\leqslant K+4$ and $j_0\leqslant K+4$ such that:
\begin{itemize}
 \item $\alpha_{i_0}\alpha_{i_0+1}$ is a subpath of $p$,
 \item $\beta_{j_0}\beta_{j_0+1}$ is a subpath of $p$, and 
 \item $\iota \beta_{j_0}$ lies in $\alpha_{i_0}\setminus\{\tau \alpha_{i_0}\}$.
\end{itemize}
The last property can be attained by an index shift of up to 2, since the concatenation of two consecutive $\alpha_i$ cannot be a subpath of one $\beta_j$ because the paths $\alpha_X$ and $\beta_X$ represent geodesics in $Y$, and the symmetric statement holds for $\beta_j$ and $\alpha_i$. (Hence, our upper bound for the indices is $K+4$ instead of $K+2$.)

Consider $i\in\{i_0,i_0+1\}$ and $j\in\{j_0,j_0+1\}$ for which there exists a path  $q$ of length $>0$ in $\alpha_{i}\cap\beta_{j}$. There exist lifts of $q$ in $\Gamma$ via $\alpha_i\mapsto\alpha_{i,\Gamma}$ and via $\beta_j\mapsto\beta_{j,\Gamma}$. 
Suppose these lifts are essentially equal. Then there exists a label-preserving automorphism $\phi$ of $\Gamma$ such that the lift of $q$ to a subpath of $\overline\alpha_{i,\Gamma}$ is equal to the lift of $q$ to a subpath of $\phi(\overline\beta_{j,\Gamma})$. If $\overline i\equiv i\mod 2$ and $\overline j\equiv j\mod 2$, then $x_{\overline i}$ is the initial vertex of $\overline\alpha_{i,\Gamma}$ and by $x_{\overline j}$ is the initial vertex of $\overline\beta_{j,\Gamma}$. 
Thus, there exists a path in $D$ from $\iota \alpha$ to $\iota \beta$ whose label is freely equal to a word of the form
$$\ell(x_1\to y_1)\ell(x_2\to y_2)\ell(x_1\to y_1)\dots \ell(x_{\overline i}\to\phi(x_{\overline j})) \dots \ell(y_2\to x_2)\ell(y_1\to x_1),$$
where no more than $2K+9$ factors occur. (See also Figure~\ref{figure:wpd_gr7}.) If the label-preserving automorphism groups of the components of $\Gamma$ containing $x_1$ and $x_2$ are finite, then there exist only finitely many elements of $G(\Gamma)$ represented by words of this form. Thus, we conclude that in this case, there are only finitely many possibilities for $h$. 

It remains to prove the case that, for every $i\in\{i_0,i_0+1\}$ and every $j\in\{j_0,j_0+1\}$, whenever $q$ is a path in $\alpha_i\cap\beta_j$, then the induced lifts of $q$ are essentially distinct. Note that in this case, $q$ is a piece.

By the choice of $i_0$ and $j_0$, $\alpha_{i_0}\cap\beta_{j_0}$ contains a maximal path $q$ of length $>0$ such that $q$ is an initial subpath of $\beta_{j_0}$. Since $\beta_{j_0}$ is not a piece, $\beta_{j_0}$ is not a subpath of $\alpha_{i_0}$. 
By the same argument, $\alpha_{i_0+1}$ is not a subpath of $\beta_{j_0}$, whence $\beta_{j_0}$ is a subpath of $\alpha_{i_0}\alpha_{i_0+1}$. Similarly, it follows that $\alpha_{i_0+1}$ is a subpath of $\beta_{i_0}\beta_{i_0+1}$. Hence, both $\alpha_{i_0+1}$ and $\beta_{j_0}$ are concatenations of no more than two pieces. 

We now invoke the last two conclusions of Lemma~\ref{lem:gr7_wpd}, which imply that there exist at most two possibilities for the reduced path $\overline\alpha_{i_0+1,\Gamma}$, and at most two possibilities for the reduced path $\overline\beta_{j_0,\Gamma}$. There exist initial subpaths $q_1$ of $\overline\alpha_{i_0+1,\Gamma}$ and $q_2$ of $\overline\beta_{j_0,\Gamma}$ such that we may represent $h$ by a word
$$\ell(x_1\to y_1)\ell(x_2\to y_2)\ell(x_1\to y_1)\dots \ell(q_1)\ell(q_2^{-1})\dots\ell(y_2\to x_2)\ell(y_1\to x_1),$$
with at most $2K+9$ factors, whence also in this case, there exist only finitely many possibilities for $h$.
\end{proof}

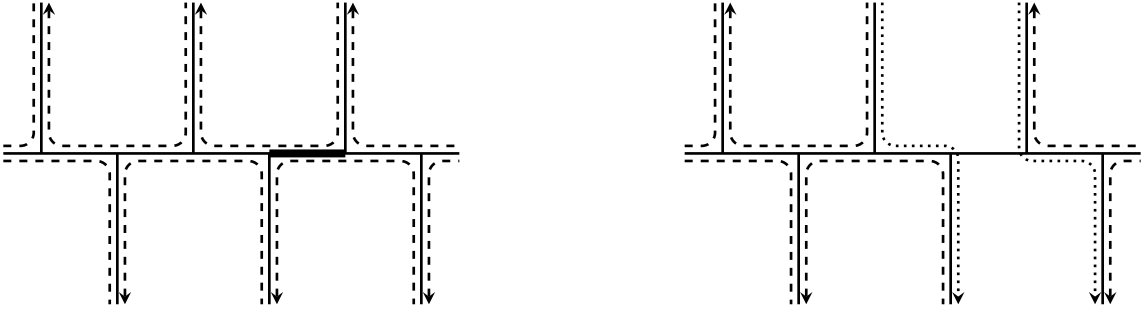
\begin{figure}\label{figure:wpd_gr7}
\begin{tikzpicture}[line width=1pt,>=stealth]
 \draw[-,line width=1pt] (-.5,0) to (5.5,0);
 \draw[-,line width=3pt] (3,0) to (4,0);
 \draw[-] (0,2) to (0,0);
 \draw[-] (2,2) to (2,0);
 \draw[-] (4,2) to (4,0);
% \draw[-] (6,2) to (6,0);
 
 \draw[-] (1,-2) to (1,0);
 \draw[-] (3,-2) to (3,0);
 \draw[-] (5,-2) to (5,0);

  \draw [-,rounded corners=5,dashed] (-.5,.1) to (-.1,.1) to (-.1,2);

 \draw [<-,rounded corners=5,dashed] (.1,2) to (.1,.1) to (1.9,.1) to (1.9,2);
 \draw [<-,rounded corners=5,dashed] (2.1,2) to (2.1,.1) to (3.9,.1) to (3.9,2);
 \draw [<-,rounded corners=5,dashed] (4.1,2) to (4.1,.1) to (5.5,.1);
 
 \draw [-,rounded corners=5,dashed] (-.5,-.1) to (.9,-.1) to (.9,-2);
 \draw [<-,rounded corners=5,dashed] (1.1,-2) to (1.1,-.1) to (2.9,-.1) to (2.9,-2);
 \draw [<-,rounded corners=5,dashed] (3.1,-2) to (3.1,-.1) to (4.9,-.1) to (4.9,-2);
 \draw [<-,rounded corners=5,dashed] (5.1,-2) to (5.1,-.1) to (5.5,-.1);
% \draw [->, dotted] (1,1) to (2,5);
 
\end{tikzpicture}
\hfill
\begin{tikzpicture}[line width=1pt,>=stealth]
 \draw[-,line width=1pt] (-.5,0) to (5.5,0);
% \draw[-,line width=3pt] (1,0) to (2,0);
 \draw[-] (0,2) to (0,0);
 \draw[-] (2,2) to (2,0);
 \draw[-] (4,2) to (4,0);
% \draw[-] (6,2) to (6,0);
 
 \draw[-] (1,-2) to (1,0);
 \draw[-] (3,-2) to (3,0);
 \draw[-] (5,-2) to (5,0);

  \draw [-,rounded corners=5,dashed] (-.5,.1) to (-.1,.1) to (-.1,2);

% \draw [->,rounded corners=5,dashed] (.1,2) to (.1,.1) to (1.9,.1) to (1.9,2);
 \draw [<-,rounded corners=5,dashed] (0.1,2) to (0.1,.1) to (1.9,.1) to (1.9,2);
 \draw [<-,rounded corners=5,dashed] (4.1,2) to (4.1,.1) to (5.5,.1);
 
 \draw [-,rounded corners=5,dashed] (-.5,-.1) to (.9,-.1) to (.9,-2);
% \draw [->,rounded corners=5,dashed] (1.1,-2) to (1.1,-.1) to (2.9,-.1) to (2.9,-2);
 \draw [<-,rounded corners=5,dashed] (1.1,-2) to (1.1,-.1) to (2.9,-.1) to (2.9,-2);
 \draw [<-,rounded corners=5,dashed] (5.1,-2) to (5.1,-.1) to (5.5,-.1);
% \draw [->, dotted] (1,1) to (2,5);

%  \draw [->,rounded corners=5,dotted] (.1,2) to (.1,-.1) to (.9,-.1) to (.9,-2);
%  \draw [->,rounded corners=5,dotted] (.9,-2) to (.9,.1) to (1.9,.1) to (1.9,2);
  \draw [->,rounded corners=5,dotted] (3.9,2) to (3.9,-.1) to (4.9,-.1) to (4.9,-2);
  \draw [->,rounded corners=5,dotted] (2.1,2) to (2.1,.1) to (3.1,.1) to (3.1,-2);
  
%  (.1,2) to (.1,.1) to (1.9,.1) to (1.9,2);

\end{tikzpicture}
\caption{
The horizontal line represents the intersection $\alpha\cap\beta$ in $D$. The vertical lines are for illustration only, providing support for the dashed paths, which lift to paths $\overline\alpha_{i,\Gamma}:x_{\overline i}\to y_{\overline i}$. 
If the $q$ traverses the thick part in the left-hand picture and if the induced lifts of $q$ are essentially equal, then the dotted paths in the right-hand picture lift to paths $x_{\overline i}\to\phi(x_{\overline j})$, respectively $y_{\overline i}\to\phi(y_{\overline j})$ in $\Gamma$ for some label-preserving automorphism $\phi$ of $\Gamma$. 
}
\end{figure}

\subsection{The graphical $Gr'(\frac{1}{6})$-case}\label{subsection:1over6}

In the presence of the $Gr'(\frac{1}{6})$-condition, we can drop all finiteness assumptions:

\begin{thm}\label{thm:gr16} Let $\Gamma$ be a $Gr'(\frac{1}{6})$-labelled graph that has at least two non-isomorphic components that each contain a simple closed path of length at least 2. Then $G(\Gamma)$ is either virtually infinite cyclic or acylindrically hyperbolic.
\end{thm}

We will rely on the following adaption of Lemma~\ref{lem:gr7_wpd} to define our WPD element as before.

\begin{lem}\label{lem:gr16_wpd} Let $\Gamma$ be a $Gr'(\frac{1}{6})$-labelled graph that has at least two non-isomorphic (not necessarily finite) components $\Gamma_1$ and $\Gamma_2$ that each contain a simple closed path of length at least 2. Then there exist vertices $x_1,y_1$ in $\Gamma_1$ and $x_2,y_2$ in $\Gamma_2$ for which the conclusion of Lemma~\ref{lem:gr7_wpd} holds.
\end{lem}

\begin{proof}
 Denote $X:=\Cay(G(\Gamma),S)$. For each $i\in\{1,2\}$, let $\gamma_i$ be a simple closed path in $\Gamma_i$ of minimal length greater than 1, and denote by $v_i$ the initial vertex of $\gamma_i$. Consider the maps $f_i:\Gamma_i\to X$ that send $v_i$ to $1$, and denote $C:=f_1(\Gamma_1)\cap f_2(\Gamma_2)$. 
 For each $i$, let $w_i$ be a vertex in $\gamma_i\subseteq \Gamma_i$ for which $d(w_i,f_i^{-1}(C))$ is maximal. Since $|\gamma_i|\geqslant 2$ and since any path in $C$ is a piece, we have $w_i\notin f_i^{-1}(C)$ by the small cancellation condition.
 
 Let $i\in\{1,2\}$, and suppose there exists a path $p$ in $\Gamma_i$ with $\iota p=w_i$ and $\tau p\in f_i^{-1}(C)$ that is a concatenation of at most 2 pieces. Choose such a $p$ with minimal length. Then $p$ is a simple path. Let $q$ be a shortest path in $\gamma_i$ with $\iota q=w_i$ and $\tau q\in f_i^{-1}(C)$. %If $\tau p=\tau q$, denote by $c$ the empty path. 
 Since $C$ is connected by Lemma~\ref{lem:connected_embedding}, there exists a shortest path $c$ in $f_i^{-1}(C)$ with $\iota c=\tau p$ and $\tau c=\tau q$ which, as observed above, is a piece. 
 
 If $pcq^{-1}$ is a non-trivial closed path, then there exists a subpath $\gamma'$ of its reduction that is a simple closed path of length at least 2. The path $\gamma'$ can be written as a concatenation of at most 3 pieces and a subpath of $q^{-1}$. Since $|\gamma'|\geqslant|\gamma_i|$ and since $|q|\leqslant \lfloor\frac{|\gamma_i|}{2}\rfloor$, this is a contradiction to the small cancellation assumption.
 
 If $pcq^{-1}$ is a trivial closed path, then $|c|=0$ and $p=q$. Now there exists a simple path $q'$ in $\gamma_i$ such that $\iota q'=w_i$, $\tau q'\in f_i^{-1}(C)$ and such that $q$ and $q'$ are edge-disjoint. Lt $c'$ be a simple path in $f_i^{-1}(C)$ with $\iota c'=\tau q$ and $\tau c'=\tau q'$. Then $\gamma'':=qc'q'^{-1}$ is a simple closed path.
 Note that $|q'|\leqslant |q|+1$. Thus, if $|c'|>0$, then $|qc'|\geqslant\frac{|\gamma_i''|}{2}$, which, together with the fact that $qc'$ is a concatenation of at most 3 pieces, contradicts the small cancellation assumption. If, on the other hand, $|c'|=0$, then the fact that $q$ is a concatenation of at most 2 pieces yields that $|q|<\frac{2|q|+1}{3}$, which cannot hold since $|q|\geqslant 1$.
 
 We conclude for $x_1=w_1,y_1=v_1,x_2=v_2,y_2=w_2$ as in Lemma~\ref{lem:gr7_wpd}.
\end{proof}

To remove the requirement that the components containing $x_1$ and $x_2$ have finite automorphism groups, which we use in the proof of Proposition~\ref{prop:wpd7}, we prove the following:

\begin{lem}\label{lem:gr16} Let $\Gamma_1$ and $\Gamma_2$ be non-isomorphic components of a $Gr'(\frac{1}{6})$-labelled graph. Suppose there exist vertices $x_1,y_1\in \Gamma_1$ and $x_2,y_2\in \Gamma_2$, such that that $x_1\neq y_1$ and such that no path from $x_1$ to $y_1$ is a concatenation of at most two pieces. Let $\phi:\Gamma_2\to\Gamma_2$, $\phi_1:\Gamma_1\to\Gamma_1$, and $\phi_2:\Gamma_2\to\Gamma_2$ be label-preserving automorphisms such that:
\begin{itemize}
 \item There exist paths  $q_2:y_2\to \phi(y_2)$ and $p_1:x_1\to\phi_1(x_1)$ such that $q_2$ and $p_1$ have the same label.
 \item There exist paths $q_1:y_1\to \phi_1(y_1)$ and $p_2:x_2\to\phi_2(x_2)$ such that $q_1$ and $p_2$ have the same label. 
\end{itemize}
Then $\phi$, $\phi_1$, and $\phi_2$ are the identity maps.\end{lem}

\begin{proof}
 Assume $\phi_1$ is non-trivial. By assumption, for every $k$ there exist paths $p^{(k)}:x_1\to \phi_1^k (x_1)$ and $q^{(k)}:y_1\to \phi_1^k (y_1)$ that are pieces and whose labels are $k$-th powers of a freely non-trivial word each. Let $\gamma$ be a geodesic $x_1\to y_1$. 
 Suppose $\phi_1$ has infinite order. By assumption, $p^{(k)}$ and $q^{(k)}$ do not intersect whence, for $k$ large enough, the reduction of $p^{(k)}\phi^k(\gamma)(q^{(k)})^{-1}\gamma^{-1}$ contains a simple closed path that contradicts the $Gr'(\frac{1}{6})$-condition. Therefore, $\phi_1$ has finite order $K$. But in this case, the path $p_1\phi_1(p_1)\phi_1^2(p_1)\dots\phi_1^{K-1}(p_1)$ is a non-trivial closed path whose label is a  piece, a contradiction. Therefore, $\phi_1$ is trivial.

This implies that $y_2$ is connected to $\phi(y_2)$ by a path of length 0 and $x_2$ is connected to $\phi_2(x_2)$ by a path of length 0, whence these two automorphisms are trivial as well.
\end{proof}

\begin{proof}[Proof of Theorem~\ref{thm:gr16}]

We define $g$ as before to be the element of $G(\Gamma)$ represented by $\ell(x_1\to y_1)\ell(x_2\to y_2)$, where the $x_i$ and $y_i$ are those produced by Lemma~\ref{lem:gr16_wpd}. Then the statement and proof of Lemma~\ref{lem:c7_geodesic} clearly apply to $g$. This shows hyperbolicity of $g$.

To prove the WPD condition, consider the proof of Proposition~\ref{prop:wpd7}, and choose the constant $N_0$ such that $d_Y(1,g^N)>2K+7$ for all $N\geqslant N_0$. The only ingredient in the proof of Proposition~\ref{prop:wpd7} that is not present in the case of Theorem~\ref{thm:gr16} is the finiteness of the automorphism groups of the components of $\Gamma$. This ingredient is used exclusively in the following case of claim 2: 
There exists $i\leqslant K+5,j\leqslant K+5$ and a path $q_i$ in $\alpha_i\cap\beta_j$ such that the lifts of $q$ via $\alpha_i\mapsto\alpha_{i,\Gamma}$ and via $\beta_j\mapsto\beta_{j,\Gamma}$ are essentially equal. In this case, we may represent $h$ by a word
$$\ell(x_1\to y_1)\ell(x_2\to y_2)\ell(x_1\to y_1)\dots \ell(x_{\overline i}\to\phi(x_{\overline j}))\dots \ell(y_2\to x_2)\ell(y_1\to x_1),$$
with at most $2K+9$ factors. Since $x_1$ and $x_2$ are contained in non-isomorphic components of $\Gamma$ we have $\overline i=\overline j$.

By our choice of $N_0$, the paths $\alpha_i\alpha_{i+1}\alpha_{i+2}$ and $\beta_j\beta_{j+1}\beta_{j+2}$ are subpaths of $p$. Using arguments of claim 2 in the proof of Lemma~\ref{lem:c7_geodesic} it now follows that there exists a path $q'$ in $\alpha_{i+1}\cap\beta_{j+1}$ for which the two resulting lifts are essentially equal and  that there exists $q''$ in $\alpha_{i+2}\cap\beta_{j+2}$ for which the lifts are essentially equal.  
Therefore, we are in the situation of Lemma~\ref{lem:gr16} where, if $\overline i\not\equiv 2\mod 2$, the indices $1$ and $2$ in the statement of the lemma have to be switched. (See also Figure~\ref{figure:wpd_gr16} for an illustration.) Therefore, $\phi$ is the identity and, in this case, $h$ is equal to $g^{\frac{i-j}{2}}$, whence finiteness is proved.
\end{proof}

\begin{figure}\label{figure:wpd_gr16}
\begin{tikzpicture}[line width=1pt,>=stealth]
 \draw[-,line width=1pt] (-.5,0) to (5.5,0);
 \draw[-,line width=3pt] (3,0) to (4,0);
 \draw[-] (0,2) to (0,0);
 \draw[-] (2,2) to (2,0);
 \draw[-] (4,2) to (4,0);
% \draw[-] (6,2) to (6,0);
 
 \draw[-] (1,-2) to (1,0);
 \draw[-] (3,-2) to (3,0);
 \draw[-] (5,-2) to (5,0);

  \draw [-,rounded corners=5,dashed] (-.5,.1) to (-.1,.1) to (-.1,2);

 \draw [<-,rounded corners=5,dashed] (.1,2) to (.1,.1) to (1.9,.1) to (1.9,2);
 \draw [<-,rounded corners=5,dashed] (2.1,2) to (2.1,.1) to (3.9,.1) to (3.9,2);
 \draw [<-,rounded corners=5,dashed] (4.1,2) to (4.1,.1) to (5.5,.1);
 
 \draw [-,rounded corners=5,dashed] (-.5,-.1) to (.9,-.1) to (.9,-2);
 \draw [<-,rounded corners=5,dashed] (1.1,-2) to (1.1,-.1) to (2.9,-.1) to (2.9,-2);
 \draw [<-,rounded corners=5,dashed] (3.1,-2) to (3.1,-.1) to (4.9,-.1) to (4.9,-2);
 \draw [<-,rounded corners=5,dashed] (5.1,-2) to (5.1,-.1) to (5.5,-.1);
% \draw [->, dotted] (1,1) to (2,5);
 
\end{tikzpicture}
\hfill
\begin{tikzpicture}[line width=1pt,>=stealth]
 \draw[-,line width=1pt] (-.5,0) to (5.5,0);
 \draw[-,line width=3pt] (1,0) to (2,0);
 \draw[-] (0,2) to (0,0);
 \draw[-] (2,2) to (2,0);
 \draw[-] (4,2) to (4,0);
% \draw[-] (6,2) to (6,0);
 
 \draw[-] (1,-2) to (1,0);
 \draw[-] (3,-2) to (3,0);
 \draw[-] (5,-2) to (5,0);

  \draw [-,rounded corners=5,dashed] (-.5,.1) to (-.1,.1) to (-.1,2);

% \draw [->,rounded corners=5,dashed] (.1,2) to (.1,.1) to (1.9,.1) to (1.9,2);
 \draw [<-,rounded corners=5,dashed] (0.1,2) to (0.1,.1) to (1.9,.1) to (1.9,2);
 \draw [<-,rounded corners=5,dashed] (4.1,2) to (4.1,.1) to (5.5,.1);
 
 \draw [-,rounded corners=5,dashed] (-.5,-.1) to (.9,-.1) to (.9,-2);
% \draw [->,rounded corners=5,dashed] (1.1,-2) to (1.1,-.1) to (2.9,-.1) to (2.9,-2);
 \draw [<-,rounded corners=5,dashed] (1.1,-2) to (1.1,-.1) to (2.9,-.1) to (2.9,-2);
 \draw [<-,rounded corners=5,dashed] (5.1,-2) to (5.1,-.1) to (5.5,-.1);
% \draw [->, dotted] (1,1) to (2,5);

%  \draw [->,rounded corners=5,dotted] (.1,2) to (.1,-.1) to (.9,-.1) to (.9,-2);
%  \draw [->,rounded corners=5,dotted] (.9,-2) to (.9,.1) to (1.9,.1) to (1.9,2);
  \draw [->,rounded corners=5,dotted] (3.9,2) to (3.9,-.1) to (4.9,-.1) to (4.9,-2);
  \draw [->,rounded corners=5,dotted] (2.1,2) to (2.1,.1) to (3.1,.1) to (3.1,-2);
  
%  (.1,2) to (.1,.1) to (1.9,.1) to (1.9,2);

\end{tikzpicture}
\caption{
The horizontal line represents the intersection $\alpha\cap\beta$ in $D$. The vertical lines are for illustration only, providing support for the dashed paths, which lift to paths $\overline\alpha_{i,\Gamma}:x_{\overline i}\to y_{\overline i}$. 
If the $q$ traverses the thick part in the left-hand picture and if the lifts of $q$ are essentially equal, then the dotted paths in the right-hand picture lift to paths $x_{\overline i}\to\phi(x_{\overline i})$, respectively $y_{\overline i}\to\phi(y_{\overline i})$ in $\Gamma$ for some label-preserving automorphism $\phi$ of $\Gamma$. 
Hence, the properties of the $\overline\alpha_{i,\Gamma}$ imply the path $q'$ traversing the thick part in the right-hand picture cannot be a piece, whence the lifts of $q'$ are also essentially equal. Since $\alpha\cap\beta$ is long enough, we have at least 3 consecutive situations as in the figure, and we obtain the situation of Lemma~\ref{lem:gr16}.
}
\end{figure}
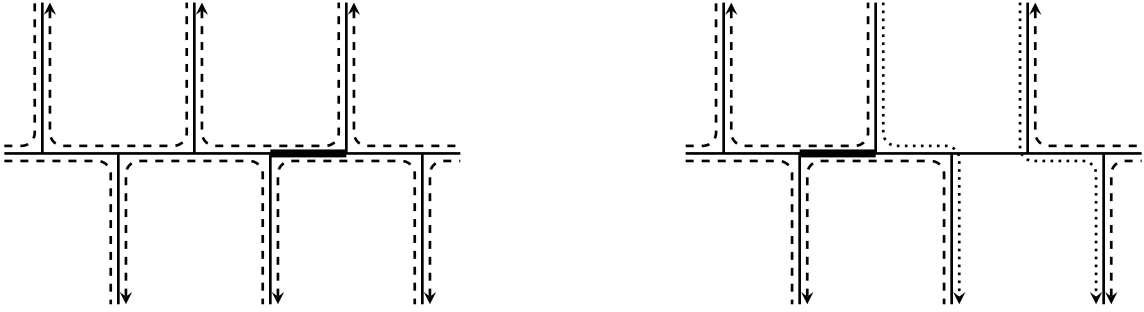

\subsection{The free product case}\label{subsection:products}

The corresponding results for groups defined by graphical free product small cancellation presentations also hold with the same proofs if we assume that at least two of the $G_i$ are non-trivial. Here ``finiteness'' means that there exists a finite graph $\Gamma'$ whose completion is $\overline \Gamma$. Equivalently, it means that there are finitely many vertices in $\overline \Gamma$ that are incident at two edges whose labels lie in distinct factors $G_i$, and that for every vertex $v$, the set of labels of edges incident at $v$ is contained in finitely many $G_i$.

\begin{thm}\label{thm:gr7*} Let $\Gamma$ be a $Gr_*(7)$-labelled graph over a free product with at least two non-trivial factors such that the components of $\Gamma$ are finite. Then $G(\Gamma)_*$ is either virtually cyclic or acylindrically hyperbolic. 
\end{thm}

\begin{thm}\label{thm:c7*} Let $\Gamma$ be a $C_*(7)$-labelled graph over a free product with at least two non-trivial factors. Then $G(\Gamma)_*$ is either virtually cyclic or acylindrically hyperbolic. 
\end{thm}

\begin{thm}\label{thm:gr16*} Let $\Gamma$ be a $Gr_*'(\frac{1}{6})$-labelled graph such that $\overline\Gamma$ contains at least two non-isomorphic components that contain closed paths whose labels are non-trivial in $\freeproduct_{i\in I}G_i$. Then $G(\Gamma)_*$ is either virtually infinite cyclic or acylindrically hyperbolic. 
\end{thm}

We explain how these results are deduced from the proofs we have already obtained in this section.

\begin{proof}[Proof of Theorems~\ref{thm:gr7*}, \ref{thm:c7*}, and \ref{thm:gr16*}] 
If $\Gamma$ is finite, then $G(\Gamma)_*$ is hyperbolic relative to the $\{G_i\mid i\in I\}$ by Theorem~\ref{thm:relatively_hyperbolic}. The proof of \cite[Lemma~2.1]{Gru} shows that $\overline\Gamma$ injects into $\Cay(G(\Gamma),\sqcup_{i\in I} S_i)$. Thus, $G(\Gamma)_*$ is non-trivially relatively hyperbolic unless the vertex set of every non-trivial component of $\overline\Gamma$ is equal to the vertex set of each one of the attached non-trivial $\Cay(G_i,S_i)$ and for every non-trivial $G_i$, $\Cay(G_i,S_i)$ is attached at every component of $\overline \Gamma$. In the case of Theorem~\ref{thm:gr7*} this can only hold if every $G_i$ is finite, in which case $G(\Gamma)_*\cong G_i$ is finite and the statement holds. In the cases of Theorems~\ref{thm:c7*} and \ref{thm:gr16*} this cannot hold at all. If $G(\Gamma)_*$ is non-trivially 
relatively hyperbolic and not virtually cyclic, then it is acylindrically hyperbolic as it acts acylindrically with unbounded orbits on the hyperbolic space $\Cay(G(\Gamma),\sqcup_{i\in I}G_i)$ by \cite[Proposition 5.2]{Os-acyl} and \cite[Corollary 4.6]{Os-hypel}.

Now assume that $\Gamma$ is infinite. We explain how to adapt the proofs from Section~\ref{section:wpd_element}. Instead of considering $\Gamma$, we must now consider $\overline\Gamma$. For simplicity, assume that each $G_i$ is non-trivial. Then, automatically, there does not exist an edge whose label occurs exactly once on the graph, and we can apply the proofs of Lemmas~\ref{lem:c7_wpd1}, \ref{lem:gr7_wpd}, and \ref{lem:c7_wpd}. Here, when considering non-trivial  closed paths or simple closed paths, 
we always require that their labels are not trivial in the free product of the $G_i$. In Lemma~\ref{lem:gr7_wpd}, we replace the claim that there exists at most one reduced path $\overline \alpha_{i,\overline\Gamma}:x_i\to y_i$ such that $\alpha_{i,\overline\Gamma}$ is a concatenation of at most two pieces by the claim that there exist at most one element of $\freeproduct_{i\in I}G_i$ represented by the labels of paths $\overline\alpha_{i,\overline\Gamma}$ 
for which $\alpha_{i,\overline\Gamma}$ is a concatenation of at most two pieces. For convenience, we denote the set of these (at most two) elements of $*_{i\in I}G_i$ by $Z$. The statements and proofs of Lemmas~\ref{lem:gr16_wpd} and \ref{lem:gr16} also apply.

Thus, we are able to define the WPD element $g$ as before, and the proof of hyperbolicity of $g$, Lemma~\ref{lem:c7_geodesic}, applies. In the proof of Proposition~\ref{prop:wpd7}, we need to make an additional observation: It is no restriction to assume that for every $i$, the terminal edge of $\alpha_i$ has a label from a different generating factor than that of the initial edge of $\alpha_{i+1}$, and to make the same assumption for every $\beta_j$ and $\beta_{j+1}$. This assumption is required since any finiteness statement only applies to vertices in the intersections of two attached Cayley graphs. We also choose $N_0$ such that $d_Y(1,g^N)>2K+6$ for all $N\geqslant N_0$. The corresponding adaption of the arguments of Proposition~\ref{prop:wpd7} occurs in the last case of the proof of claim 2. 

By our choice of $N_0$, $\alpha_{i_0}\alpha_{i_0+1}\alpha_{i_0+2}$ and $\beta_{j_0}\beta_{j_0+1}\beta_{j_0+2}$ lie in $p$, i.e.\ in the last case of the proof of claim 2, we may consider all $i\in \{i_0,i_0+1,i_0+2\}$ and $j\in \{j_0,j_0+1,j_0+2\}$. %We proceed as follows:
Every $\alpha_i$ (or $\beta_j$) under consideration is a concatenation of two pieces, but not a piece itself. Observe that, in any attached Cayley graph in $\overline\Gamma$, either every path is a piece or no path is a piece. Therefore, the label of $\alpha_i$ (or $\beta_j$) cannot lie in one of the generating free factors and, hence, $\alpha_i$ (or $\beta_j$) contains in its interior a vertex where two edges with labels from distinct free factors meet. Note that $\beta_{j_0}$ is a subpath of $\alpha_{i_0}\alpha_{i_0+1}$, $\alpha_{i_0+1}$ is a subpath of $\beta_{j_0}\beta_{j_0+1}$, and $\beta_{j_0+1}$ is a subpath of $\alpha_{i_0+1}\alpha_{i_0+2}$. 
Hence, each of these 3 paths is a concatenation of at most 2 pieces and, hence, the labels of the paths $\overline\alpha_{i_0+1,\overline\Gamma},\overline\beta_{j_0,\overline\Gamma},\overline\beta_{j_0+1,\overline\Gamma}$ all represent elements of $Z$. Consider a vertex $v$ in the interior of $\alpha_{i_0+1}$ incident at edges 
with labels from two distinct factors. Then at least one of the lifts $\beta_{j_0}\mapsto \beta_{j_0,\overline\Gamma}$ or $\beta_{j_0+1}\mapsto \beta_{j_0+1,\overline\Gamma}$ is defined on $v$ and takes $v$ to a vertex in the intersection of two edges of $\overline\beta_{j_0,\overline\Gamma}$, respectively two edges of $\overline\beta_{j_0+1,\overline\Gamma}$, with labels from distinct factors. We first assume this holds for $j_0$, and note that the proof for $j_0+1$ is completely analogous; only the final constant must be raised by 1. 

Our minimality assumptions on $\sum_{i=1}^{2N}|\overline\alpha_{i,\overline\Gamma}|$ and $\sum_{j=1}^{2N}|\overline\beta_{j,\overline\Gamma}|$ imply that the labels of $\overline\alpha_{i_0+1,\Gamma}$ and $\overline\beta_{j_0,\overline\Gamma}$ are reduced words in the free product sense. We may write the elements of $Z$ uniquely as $z=g_1g_2\dots g_{n_1}$ and $z'=g_1'g_2'\dots g_{n_2}'$, where each $g_l$ is non-trivial in some $G_{k_l}$ and for each $l$ we have $k_l\neq k_{l+1}$, and, similarly, each $g_l'$ is non-trivial in some $G_{k_l'}$ and for each $l$ we have $k_l'\neq k_{l+1}'$. Since $\alpha_{i_0+1}$ and $\beta_{j_0}$ intersect in $v$ and since, in each case, the image of $v$ in $\overline\Gamma$ lies in the intersection of edges in the paths $\overline\alpha_{i_0+1,\Gamma}$ and $\overline\beta_{j_0,\overline\Gamma}$ with labels from distinct factors, 
we may write $h$ as
$$\ell(x_1\to y_1)\ell(x_2\to y_2)\ell(x_1\to y_1)\dots w_1w_2^{-1}\dots \ell(y_2\to x_2)\ell(y_1\to x_1),$$
where each $w_i$ is an initial subword of $z$ or $z'$ as written above (of which, in particular, there are only finitely many), and where at most $2K+9$ factors occur. This completes the proof.
\end{proof}

\section{New examples of divergence functions}\label{divergence}

In this section, we construct the first examples of groups with divergence functions in the gap between polynomial and exponential functions.

We recall the definition of divergence of a geodesic metric space following \cite{DMS-div}. Let $X$ be a geodesic metric space. A \emph{curve} in $X$ is a continuous map $I\to X$, where $I$ is a compact real interval. Fix constants $0<\delta<1$, and let $\gamma\geqslant 0$. 
For a triple of points $a, b, c \in X$ with $d(c, \{a, b\}) = r > 0$, let $\mathrm{div}_{\gamma} (a, b, c; \delta)$ be the infimum of the lengths of curves from $a$ to $b$ whose images do not intersect $B_{\delta r - \gamma}(c)$, where $B_\lambda(Y)$ denotes the open ball of radius $\lambda$ around a subset $Y$ of $X$. If no such curve exists, set $\mathrm{div}_\gamma (a, b, c; \delta) = \infty$.
\begin{defi}
The divergence function $\Div^X_\gamma (n, \delta)$ of the space $X$ is defined as
the supremum of all numbers $\mathrm{div}_\gamma (a, b, c; \delta)$ with $d(a, b) \leqslant n$.
\end{defi}

If $X$ is a connected graph, then we may consider $X$ as a geodesic metric space by isometrically identifying each edge of $X$ with either the unit interval or the 1-sphere. With this identification, every path gives rise to a curve.

For functions $f,g:\R^+\to\R^+$ we write $f\preceq g$ if there exists $C>0$ such that for every $n\in \R^+$, $f(n)\leqslant Cg(Cn+C)+Cn+C$, and define $\succeq,\asymp$ similarly. By \cite[Corollary 3.12]{DMS-div}, if $X$ is a Cayley graph then we have $\Div^X_\gamma (n, \delta)\asymp \Div^X_{2} (n, 1/2)$ whenever $0<\delta\leqslant 1/2$ and $\gamma\geqslant 2$. 
Also, the $\asymp$-equivalence class of $\Div^X_\gamma (n, \delta)$ is a quasi-isometry invariant (of Cayley graphs). Given a group $G$ with a specified finite generating set, we write $\Div^G(n)$ for $\Div^X_{2} (n, 1/5)$, where $X$ is the Cayley graph realized as geodesic metric space. 

\begin{thm}\label{thm:divalt}
 Let $r_N:=(a^{N}b^{N}a^{-N}b^{-N})^4$, and for $I\subseteq \N$, let $G(I)$ be defined by the presentation $\langle a,b\mid r_i,i\in I\rangle$. Then, for every infinite set $I\subseteq\N$ we have:
\begin{equation}\label{eqn:quadratic_a}
\liminf_{n\to\infty} \frac{\Div^{G(I)}(n)}{n^2}<\infty.
\end{equation}

 Let $\{f_k\mid k\in \N\}$ be a countable set of subexponential functions. Then there exists an infinite set $J\subseteq \N$ such that for every function $g$ satisfying $g\preceq f_k$ for some $k$ we have for every subset $I\subseteq J$:
 
 \begin{equation}\label{eqn:supersubexponential_a}
\limsup_{n\to\infty}\frac{\Div^{G(I)}(n)}{g(n)}=\infty.
\end{equation}
\end{thm}

The set of relators $\{r_1,r_2,\dots\}$ satisfies the classical $C'(\frac{1}{6})$-condition. Thus, the groups constructed in this theorem are acylindrically hyperbolic by Theorem~\ref{thm:gr7}. 

The idea of proof for Theorem~\ref{thm:divalt} is to use the fact that cycle graphs labelled by the relators of a classical $C'(\frac{1}{6})$-presentation are isometrically embedded in the Cayley graph. This enables us to construct detours in the Cayley graph which provide the upper (quadratic) bound, see Figure~\ref{figure:fence}. The facts that every finitely presented classical $C'(\frac{1}{6})$-group is hyperbolic and that hyperbolic groups have exponential divergence will be used to obtain the lower (subexponential) bound.

\begin{remark}\label{remark:uncountable}
 Let $J$ be an infinite subset of $\N$ as in the second statement of the Theorem, and let $I$ be a subset of $J$ whose elements are a sequence of superexponential growth. Then, for any $I_1,I_2\subseteq I$, $G(I_1)$ and $G(I_2)$ can only be quasi-isometric if the symmetric difference of $I_1$ and $I_2$ is finite by \cite[Proposition 1]{Bow-manyqiclasses}. 
 Hence, given the countable set of subexponential functions $\{f_k\mid k\in\N\}$, we construct an uncountable family of pairwise non-quasi-isometric groups whose divergence functions satisfy the conclusion of Theorem~\ref{thm:divalt}.
\end{remark}

We first prove the second claim of Theorem~\ref{thm:divalt}. We collect useful facts:

\begin{lem}\label{lem:divfacts}$\,$
\begin{enumerate}[(i)]
 \item The divergence function of a $\delta$-hyperbolic Cayley graph is bounded below by $n\mapsto 2^{(n/5-3)/\delta}$.
 \item If $G=\langle S\mid R\rangle$ is a finite $C'(\frac{1}{6})$-presentation, then the hyperbolicity constant $\delta$ of the Cayley graph of $G$ is bounded above by $2\max_{r\in R}|r|$.
 \item Suppose that $G$ has the $C'(\frac{1}{6})$-presentation $\langle S\mid R\cup R'\rangle$. Also, suppose that each $r\in R'$ satisfies $|r|\geqslant 4N$ for some integer $N$. Then the Cayley graph of $G$ (with respect to $S$) contains an isometric copy of an $N$-ball in the Cayley graph of $\langle S\mid R\rangle$.
 
\end{enumerate}
\end{lem}

\begin{proof}
 $(i)$ This is an easy consequence of \cite[Proposition III.H.1.6]{BrHa-metspaces}. In fact, if $c$ is the midpoint of a geodesic of length $n$ from $a$ to $b$ and $\alpha$ is any path from $a$ to $b$ not intersecting $B_{n/5-2}(c)$ then
$$n/5-2\leqslant d(c,\alpha)\leqslant \delta \log_2 |\alpha|+1,$$
whence $|\alpha|\geqslant 2^{(n/5-3)/\delta}$, as required.

$(ii)$ This easily follows from \cite[Theorem 43]{Str}.

$(iii)$ If not, there is a non-geodesic path $\gamma$ of length at most $2N$ in the Cayley graph of $G$ whose label represents a geodesic in the Cayley graph of $\langle S\mid R\rangle$. A loop whose boundary is the union of $\gamma$ and a geodesic in $G$ encloses a van Kampen diagram of boundary length at most $4N-1$ that must involve one of the relators from $R'$. However, such van Kampen diagram does not exist by \cite[Lemma 13-(3)]{Oll}, which says that the length of the boundary of a reduced van Kampen diagram for a $C'(\frac{1}{6})$-presentation is at least as large as the length of the relators it contains.
\end{proof}

\begin{proof}[Proof of Equation~$\ref{eqn:quadratic_a}$]

We recursively define the set $J=\{j_1,j_2,\dots\}$. Let $\{g_k\mid k\in\N\}$ be an enumeration of all functions of the form $t\mapsto C f_m(Ct+C)+Ct+C$, where $m,C\in \N$. First, choose $j_1$ arbitrary. Then,  for $n\geqslant 1$, suppose we have chosen $J_n:=\{j_1,j_2,\dots,j_n\}$, and let $G_n:=G(J_n)$. Let $\delta_n$ be the hyperbolicity constant of $G_n$.

Since every $g_k$ is subexponential, for each sufficiently large $N$ we have for every $k\leqslant n$:
\begin{equation}\label{eqn:divergence-hyperbolicity}
g_k(N)<\frac{1}{N}2^{(N/5-3)/(2\rho_n)}\leqslant\frac{1}{N}2^{(N/5-3)/\delta_n}\leqslant\frac{1}{N} \Div^{G_n}(N),
\end{equation}
where $\rho_n$ is the length of the longest relator in the presentation of $G_n$. This uses Lemma~\ref{lem:divfacts} (i) and (ii). Furthermore, employing Lemma~\ref{lem:divfacts}(iii) we obtain that if $|r_{j_{i}}|\geqslant 4\left(2^{(N/5-3)/(2\rho_n)}\right)$ for every $i>n$, then
\begin{equation*}
g_k(N)<\frac{1}{N}2^{(N/5-3)/(2\rho_n)}\leqslant\frac{1}{N} \Div^{G(J)}(N).
\end{equation*}
Therefore, we choose $r_{j_{n+1}}$ of length at least $4\left(2^{(N/5-3)/(2\rho_n)}\right)$. We proceed inductively to define $J$, letting the numbers $N$ in the construction go to infinity. Then we have for every $g_k$:
$$\limsup_{n\to\infty} \frac{\Div^{G(J)}(n)}{g_k(n)}=\infty.$$
If $I$ is a subset if $J$, then Inequality~\ref{eqn:divergence-hyperbolicity} still holds for $\rho_n$ unchanged (i.e. the length of the longest relator in $J_n$), and $\delta_n$ and $G_n$ defined by the set of relators $I\cap J_n$. Again, the estimate for the divergence function at $N$ carries over to $G(I)$.
\end{proof}

The following will enable us to prove Equation~$\ref{eqn:supersubexponential_a}$:

\begin{prop}\label{prop:divergence} Let $G$ be defined by a classical $C'(\frac{1}{6})$-presentation $\langle a,b\mid R\rangle$, and let $X=\Cay(G,\{a,b\})$. Let $n,N\in \N$ such that $N\geqslant 2n$. Suppose $r_N\in R$. Let $x,y,m$ be vertices in $X$ with $0<d(x,y)\leqslant n$, with $r:=d(x,m)\leqslant d(y,m)$, and with $r>0$. Then there exists a path from $x$ to $y$ of length at most $20nN+32N$ that does not intersect $B_{r/5}(m)$.
\end{prop}

\begin{proof}
 Let $g$ be a geodesic path from $x$ to $y$. If $g$ does not intersect $B:=B_{r/5}(m)$ in a vertex, then the statement holds. Hence we assume $B\cap g$ contains a vertex. Since $g\subseteq \overline B_{n/2}(\{x,y\})$, we obtain $r/5+n/2\geqslant r$, whence $r\leqslant 5n/8$.
 
 Let $g'$ be the shortest initial subpath of $g$ that terminates at a vertex of $B$. Then $|g'|>(4/5)r$. 
 Thus, if $g=g'g''$, then $|g''|<n-(4/5)r$, whence $d(y,m)<n-(4/5)r+r/5=n-(3/5)r$. 
 Let $g_1$ be a geodesic path from $x$ to $m$ and $g_2$ a geodesic path from $m$ to $y$, and $\sigma$ the simple path from $x$ to $y$ obtained as the reduction of $g_1g_2$ (i.e. by removing any backtracking). Then $|\sigma|<n+(2/5)r$, and every vertex in $\sigma$ has distance less than $n-(3/5)r$ from $m$.
 We decompose $\sigma$ into subpaths $\sigma=\sigma_1\sigma_2\dots\sigma_k$, where $k\leqslant |\sigma|$, such that each $\sigma_i$ is a maximal subpath whose label is a power of a generator. Note that each $\sigma_i$ satisfies $|\sigma_i|< n+(2/5)r$.

 Denote by $\Omega$ the set of all simple closed paths in $X$ that are labelled by $r_N$. By Lemmas~\ref{lem:convex_embedding} and \ref{lem:connected_embedding}, for every $\gamma,\gamma'\in \Omega$, the image of $\gamma$ is an isometrically embedded cycle graph, and $\gamma\cap\gamma'$ is either empty or connected. A \emph{block} in $\gamma\in\Omega$ is a maximal subpath that is labelled by a power of a generator.
 
 Note that $N\geqslant 2n>n+(2/5)r>|\sigma_i|$ for each $i$. Therefore, we can choose $\gamma_1\in \Omega$ such that $\sigma_1$ is an initial subpath of a block of $\gamma_1$. Then we can choose $\gamma_2\in\Omega$ such that $\sigma_2$ is an initial subpath of a block of $\gamma_2$, and such that $|\gamma_1\cap\gamma_2|\geqslant N-n-(2/5)r\geqslant 2n-n-(2/5)r\geqslant n-(2/5)r$. Iteratively, we can find a sequence $\gamma_1,\gamma_2,\dots,\gamma_k$ of elements of $\Omega$ with these properties.
 
 Denote by $\delta_1$ the maximal initial subpath of $\gamma_1^{-1}$ that does not contain edges of $\gamma_2$. Denote by $\delta_k$ the maximal initial subpath of $\gamma_k^{-1}$ that does not contain edges of $\gamma_{k-1}$. For $1<i<k$, denote by $\delta_i$ the maximal subpath of $\gamma_i^{-1}$ that does not contain edges of $\sigma_i\cup\gamma_{i-1}\cup\gamma_{i+1}$. Then $\delta:=\delta_1\delta_2\dots\delta_k$ is a path in $X$ from $x$ to $y$. See Figure~\ref{figure:fence} for an illustration.
 
 Let $1<i<k$. Then $\gamma_i=\sigma_ip\delta_i^{-1}p'$ for paths $p,p'$ with $|p|\geqslant n-(2/5)r$ and $|p'|\geqslant n-(2/5)r$. Since $\gamma_i$ is an isometrically embedded cycle graph, this implies that every vertex in $\delta_i$ has distance at least $n-(2/5)r$ from $\sigma_i$. Since $\sigma_i$ contains a vertex at distance at most $n-(3/5)r$ from $m$, every vertex in $\delta_i$ has distance at least $n-(2/5)r-n+(3/5)r\geqslant r/5$ by the reverse triangle inequality.

 Suppose $\delta_1$ contains a vertex $v$ at distance less than $r/5$ from $m$. Since $d(x,m)\geqslant r$, we must have $d(x,v)>(4/5)r$. There exists $\gamma_0\in \Omega$ such that $\gamma_0\cap\sigma_1$ is a vertex and such that $|\gamma_0\cap\gamma_1|\geqslant n-(2/5)r$. Let $\tilde\delta_0$ be the maximal initial subpath of $\gamma_0^{-1}$ that does not contain edges of $\gamma_1$, and let $\tilde\delta_1$ be the maximal terminal subpath of $\delta_1$ that does not contain edges of $\gamma_0$. 
 
 Suppose $\tilde\delta_0$ contains a vertex $v'$ with $d(v',m)<r/5$. Then, by the triangle inequality, $d(v,v')<(2/5)r$. By the above arguments, $v$ must lie in the image of an initial subpath of $\delta_0$ of length less than $n-(2/5)r$, and $v'$ must lie in the image of an initial subpath of $\tilde\delta_0$ of length less than $n-(2/5)r$. Therefore, there exists a subpath $p$ of a cyclic shift of $\gamma_0$ from $v$ to $v'$ of length less than $2n-(4/5)r$. Since $|r_N|=4N\geqslant8n$, this is a geodesic path. Note that $p$ contains $x$. Since $d(v,x)>(4/5)r$ and $d(v',x)>(4/5)r$, this implies $d(v,v')>(8/5)r$, a contradiction.

 In the case that $\delta_k$ contains a vertex $v$ at distance less than $r/5$ from $m$, we can analogously choose $\gamma_{k+1}\in\Omega$ and replace $\delta_k$ by a concatenation of paths $\tilde \delta_k\tilde\delta_{k+1}$. The resulting path $\tilde\delta$ is a path that does not intersect the ball of radius $r/5$ around $m$. Its length is at most $(|\sigma|+2)|r_{N}|\leqslant (n+(2/5)r+2)(16N)\leqslant ((5/4)n+2)16N\leqslant 20nN+32N$.
\end{proof}

\begin{figure}[h]\label{figure:fence}
\begin{center}
\includegraphics{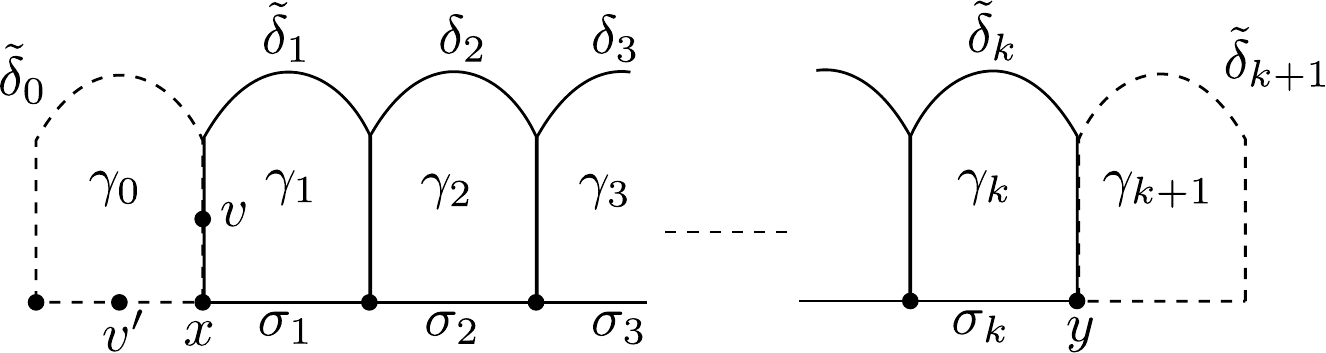}
\end{center}
\vspace{-12pt}
\caption{The construction of the path $\tilde\delta=\tilde\delta_0\tilde\delta_1\delta_2\dots\delta_{k-1}\tilde\delta_k\tilde\delta_{k+1}$ from $x$ to $y$ in the proof of Proposition~\ref{prop:divergence}.}
\end{figure}

We also consider the case that $x$, $y$, and $m$ are not necessarily vertices but possibly interior points of edges.

\begin{cor}\label{cor:divergence} Let $n\in\N$, and let $G$ be given by the a classical $C'(\frac{1}{6})$-presentation $\langle a,b\mid R\rangle$ with $r_{2n}\in R$. Then $\Div^{G}(n)\leqslant 40n^2+64n+2$.
\end{cor}

\begin{proof}
 Consider a triple of points $x,y,m$ in $X$ with $d(x,y)\leqslant n$ and $r=d(\{x,y\},m)$, where $r>0$.  Let $x'$ and $y'$ be vertices with $d(x,x')\leqslant 1,d(y,y')\leqslant 1,d(m,m')\leqslant1/2$ such that $d(x',y')\leqslant d(x,y)\leqslant n$. Then $r-2<d(\{x',y'\},m')$. By Proposition~\ref{prop:divergence} there exists a path $p$ from $x'$ to $y'$ of length at most $40n^2+64n$ (we take $N=2n$) such that $p$ does not intersect $B_{(r-2)/5}(m')$.  Note that $B_{r/5-1}(m)\subseteq B_{r/5-1/2}(m')\subseteq B_{(r-2)/5}(m')$. Therefore $p$ does not intersect $B_{r/5-1}(m)$. 
 %We may assume $r\geqslant 5$, for else $B_{r/5-1}(m)$ is empty. 
 Then $d(\{x,y\},B_{r/5-1}(m))>r-(r/5-1)>1$, whence $p$ can be extended to a path from $x$ to $y$ whose image does not intersect $B_{r/5-1}(m)$.
 \end{proof}

\begin{proof}[Proof of Equation $\ref{eqn:supersubexponential_a}$]
If $I$ is infinite, then by Corollary~\ref{cor:divergence}, $\Div^{G(I)}(n)$ is bounded from above by $40n^2+64n+2$ at infinitely many values of $n$.
\end{proof}
\section{New non-relatively hyperbolic groups}

We give a tool for constructing finitely generated groups that are not hyperbolic relative to any collection of proper subgroups. We use it to show that the groups constructed in Theorem~\ref{thm:divalt} are not non-trivially relatively hyperbolic and to construct for every finitely generated infinite group $G$ a finitely generated group $H$ that is not non-trivially relatively hyperbolic and contains $G$ as a non-degenerate hyperbolically embedded subgroup.

\begin{prop}\label{notrh}
 Let $G$ be a group with a finite generating set $S$, and denote $X:=\Cay(G,S)$. Assume that for each $K>0$ there exists a set $\Omega_K$ of isometrically embedded cycle graphs in $X$ with the following properties:
 \begin{itemize}
  \item $\cup_{\gamma\in \Omega_K}\gamma=X$, and
  \item for all $\gamma,\gamma'\in \Omega_K$ there exists a finite sequence $$\gamma=\gamma_0,\gamma_1,\dots,\gamma_n=\gamma'$$ with $\diam(\gamma_i\cap \gamma_{i+1})\geqslant K$.
 \end{itemize}
Then $G$ is not hyperbolic relative to any collection of proper subgroups.
\end{prop}

\begin{proof}
Suppose that $X$ is hyperbolic relative to a collection of subsets $\{P_i\mid i\in I\}$, and assume that $X=\bigcup_{i\in I}N_1(P_i)$. We show that there exists $i_0\in I$ and $C_1>0$ such that $X=N_{C_1}(P_{i_0})$, where $N_r$ denotes the $r$-neighborhood. This implies the proposition: If $G$ is hyperbolic relative to a collection of proper subgroups, then the Cayley graph $X$ is hyperbolic relative to the collection of the cosets of these peripheral subgroups. 
Our proof implies that a peripheral subgroup has finite index in $G$. Since $G$ is infinite by our assumptions, this contradicts the fact that peripheral subgroups of a relatively hyperbolic group are almost malnormal, see e.g. \cite{Os-rh}.

Let $\gamma$ be an isometrically embedded cycle graph in $X$. We first show that there exist a constant $C_1$, independent of $\gamma$, and $i\in I$ (which may depend on $\gamma$) such that $\gamma\subseteq N_{C_1}(P_{i})$.

There exists a simple closed path $q_1q_2q_3$ whose image is $\gamma$ such that each $q_k$ is a geodesic path of length at least $|V\gamma|/3-1$. Since $X$ is hyperbolic relative to the collection $\{P_i\mid i\in I\}$, it has the following property stated in \cite[Definition 4.31 (P)]{Drutu-qi-relhyp}: There exists constants $\sigma$ and $\delta$, independent of $\gamma$, such that there exists $i\in I$ for which $N_\sigma(P_i)$ intersects each $q_k$. 
Moreover, for each $k$ there exist vertices $x_k,y_k\in q_k\cap N_\sigma(P_i)$ (the entrance points) such that $d(x_k,y_{k+1})<\delta$ for each $k\in\{1,2,3\}$ (indices mod 3), see Figure~\ref{figure:entrance_points}. Note that we do need to consider the case \cite[Definition 4.31 (C)]{Drutu-qi-relhyp}, since the $N_1(P_i)$ cover $X$. A proof of the above property is found in \cite[Section 8]{DrutuSapir-treegraded}.

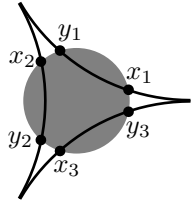
\begin{figure}\label{figure:entrance_points}
\begin{center}
\begin{tikzpicture}[>=stealth, line width=1pt]
%\draw (0,0) circle[radius=.7, dashed];
\fill[gray] (0,0) circle (.7);
%\fill [domain=0:360,gray] plot ({.7*cos(\x)}, {.7*sin(\x)});
%\node at (0,0) {\small $N_\sigma(P_i)$};
%\draw [domain=0:360,dotted] plot ({.7*cos(\x)}, {.7*sin(\x)});
\foreach \y in {1,2,3} {
\begin{scope}[rotate=120*\y-120]
\draw [domain=210:270,shift={(60:3)}] plot ({2.59*cos(\x)}, {2.59*sin(\x)});
\fill (12:.7) circle (2pt);\node at (22:.9) {\small $x_\y$};
\fill (108:.7) circle (2pt);\node at (95:.9) {\small $y_\y$};
\end{scope}
}
\end{tikzpicture}
\end{center}
\vspace{-12pt}
\caption{The cycle graph $\gamma$. The gray area represents $N_\sigma(P_i)$, and the entrance points $x_k,y_k$ are marked.}
\end{figure}

By \cite[Lemma 4.15]{DrutuSapir-treegraded}, there exists $\sigma'$ such that for every $i$, any geodesic with endpoints in $N_\sigma(P_i)$ is contained in $N_{\sigma'}(P_i)$. Let $C_1=\sigma'+2\delta$. Then, if $\diam(\gamma)\leqslant 2\delta$, we have $\gamma\subset N_{C_1}(P_i)$. If $\diam(\gamma)\geqslant 2\delta$, then we may write a cyclic conjugate of $q_1q_2q_3$ as $d_1p_1d_2p_2d_3p_3$, where $\iota d_k=x_k$, $\tau d_k=y_{k+1}$ (indices mod 3), $|d_k|<\delta$, and $p_k$ is a subpath of $q_k$ for each $k$. 
By the assumption on the diameter, each $d_k$ is a geodesic, and each $q_k$ is a geodesic since it as subpath of a geodesic. Therefore, each $d_k$ and $q_k$ is a geodesic with endpoints in $N_\sigma(P_i)$, whence it is contained in $N_{C_1}(P_i)$. Thus, $\gamma\subseteq N_{C_1}(P_i)$.

We now show that, in fact, $i$ can be chosen independently of $\gamma$: By \cite[Theorem 4.1]{DrutuSapir-treegraded}, there exists a constant $C_2$ such that for any distinct $P_i$ and $P_j$ we have $\diam(N_{C_1}(P_i)\cap N_{C_1}(P_j))< C_2$. Let $\gamma\in\Omega_{C_2}$, and let $i_0\in I$ such that $\gamma\subseteq N_{C_1}(P_{i_0})$. 
Then, for every $\gamma'\in\Omega_{C_2}$, our second assumption implies that $\gamma'\subseteq N_{C_1}(P_{i_0})$. Thus, our first assumption yields that $X=N_{C_1}(P_{i_0})$.
\end{proof}

\begin{prop}\label{prop:divgpsnotrh} If $I\subseteq \N$ is infinite, then the group $G(I)$ in Theorem~\ref{thm:divalt} is not hyperbolic relative to any collection of proper subgroups. 
\end{prop}

\begin{proof}
Since $I$ is infinite, for every $K>0$ there exists $N$ with $\min\{N-1,\lceil\frac{N}{2}\rceil\}\geqslant K$ such that $r_{N}$ is in the presentation for $G$. Let $X:=\Cay(G(I),\{a,b\})$, and let $\Omega$ be the set of embedded cycle graphs in $X$ whose label is $r_N$. A \emph{block} of such a cycle graph is a maximal subgraph that is a line graph in which every edge is labelled by the same generator. By a 1-st vertex in a block we mean a vertex in the block that is at distance 1 from one of the endpoints of the block.

Let $\gamma\in \Omega$, and let $v$ be a 1-st vertex in a block $\beta$ of $\gamma$. Let $v'$ be a vertex in $X$ at distance 1 from $v$. We show: There exists $\gamma'\in \Omega$ with $\diam(\gamma\cap\gamma')\geqslant N-1$ such that $v'$ is a 1-st vertex of a block in $\gamma'$.

Let $e$ be an edge in $\beta$ such that $\iota e=v$ and $\tau e$ is an endpoint of $\beta$, and let $e'$ be the edge in $X$ with $\iota e'=v$ and $\tau e'=v'$. Denote $s=\ell(e)$ and $s'=\ell(e')$. If $s=s'$ or $s^{-1}=s'$, then we can choose $\gamma'$ to be the translate of $\gamma$ by $s'$ under the action of $F(S)$ on $X$. 
In this case, we have $\diam(\gamma\cap\gamma')\geqslant N-1$. If $s\neq s'$ and $s^{-1}\neq s'$, then, by construction of $r_N$ there exists $\gamma'\in \Omega$ containing a path with label $s^Ns'^{N}$ such that $\gamma'\cap\gamma$ contains a path with label $s^{N-1}$ (whence $\diam(\gamma\cap\gamma')\geqslant N-1$) and such that $v'$ is a 1-st vertex in a block of $\gamma'$.

%vertex in $\gamma$ that that is the $1$-st vertex of a block $\beta$. 

%Let $s$ be the label of the edge from the endpoint of $\beta$ (that is closest to $v$) to $v$, and let $s'$ be the label of the edge $e$ with $\iota e=v$ and $\tau e=v'$. If $s=s'$ or $s^{-1}=s'$, then we can choose $\gamma'$ to be the translate of $\gamma$ by $s$, respectively $s^{-1}$, under the left-action of $F(S)$ on $X$. Then $|\gamma\cap\gamma'|\geqslant N-1$. Suppose $s\neq s'$ and $s^{-1}\neq s'$. Then, by construction of $r_N$, there exists $\gamma'\in C$ containing a path labelled by $s^{-N}(s')^{N}$ such that $\gamma'\cap\gamma$ contains a path labelled by $s^{N-1}$ and such that $v'$ is a 1-st vertex in block of $\gamma'$.

Now fix $\gamma,\gamma'\in \Omega$, and let $v$ and $v'$ be a 1-st vertices in blocks of $\gamma$, respectively $\gamma'$. Choose a path $p$ from $v$ to $v'$ with $|p|=k$. As above, we choose a sequence of cycles $\gamma=\gamma_0,\gamma_1,...\gamma_k$ such that $\diam(\gamma_i\cap\gamma_{i+1})\geqslant N-1$, and $v'$ is the 1-st vertex of a block in $\gamma_k$. If $\gamma_k=\gamma$, we are done. Now suppose $\gamma_k\neq \gamma$. 
If the two respective blocks containing $v'$ as 1-st vertex are labelled by distinct elements of $S$, then there exists $\gamma_{k+1}\in\Omega$ that intersects both $\gamma_k$ and $\gamma$ in line graphs of length $N-1$ each. If they are labelled by the same element of $S$, then there exists $\gamma_{k+1}\in\Omega$ that intersects each $\gamma_k$ and $\gamma$ in a line graph of length at least $\lceil\frac{N}{2}\rceil$, see Figure~\ref{figure:overlap_of_blocks}.
\end{proof}

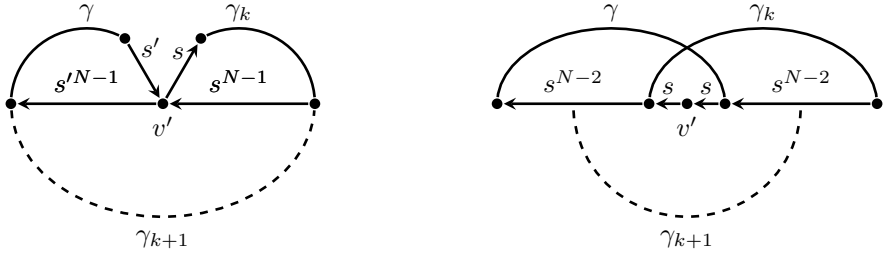
\begin{figure}\label{figure:overlap_of_blocks}
\begin{center}
\begin{tikzpicture}[>=stealth, line width=1pt,shorten <=2.5pt, shorten >=2.5pt]
%\draw (0,0) circle[radius=.7, dashed];
%\fill [domain=0:360,gray] plot ({.7*cos(\x)}, {.7*sin(\x)});
%\node at (0,0) {\small $N_\sigma(P_i)$};
%\draw [domain=0:360,dotted] plot ({.7*cos(\x)}, {.7*sin(\x)});

\draw[->] (2,0) to (0,0);
\draw[->] (0,0) to (-2,0);

\fill (2,0) circle (2pt);
\fill (0,0) circle (2pt);
\fill (-2,0) circle (2pt);
\node at (1,.3) {\small $s^{N-1}$};
\node at (-1,.3) {\small $s'^{N-1}$};

\draw [domain=0:120,shift={(1,0)}] plot ({cos(\x)}, {sin(\x)});
\draw [->] (0,0) to (60:1);
\fill (60:1) circle (2pt);

\draw [domain=60:180,shift={(-1,0)}] plot ({cos(\x)}, {sin(\x)});
\draw [<-] (0,0) to (120:1);
\fill (120:1) circle (2pt);

\draw [domain=180:360,dashed] plot ({2*cos(\x)}, {1.5*sin(\x)});

\node at (1,.3) {\small $s^{N-1}$};
\node at (.2,.67) {\small $s$};
\node at (-.16,.75) {\small $s'$};
\node at (-1,.3) {\small $s'^{N-1}$};

\node at (1,1.2) {\small $\gamma_k$};
\node at (-1,1.2) {\small $\gamma$};
\node at (0,-1.8) {\small $\gamma_{k+1}$};
\node at (0,-.3) {\small $v'$};
\end{tikzpicture}
\hspace{2cm}
\begin{tikzpicture}[>=stealth, line width=1pt,shorten <=2.5pt, shorten >=2.5pt]

\draw[->] (2.5,0) to (.5,0);
\draw[->] (.5,0) to (0,0);
\draw[->] (0,0) to (-.5,0);

\draw[->] (-.5,0) to (-2.5,0);

\fill (2.5,0) circle (2pt);
\fill (.5,0) circle (2pt);
\fill (0,0) circle (2pt);
\fill (-.5,0) circle (2pt);
\fill (-2.5,0) circle (2pt);

\draw [domain=0:180,shift={(1,0)}] plot ({1.5*cos(\x)}, {1*sin(\x)});
\draw [domain=0:180,shift={(-1,0)}] plot ({1.5*cos(\x)}, {1*sin(\x)});

\draw [domain=180:360,dashed] plot ({1.5*cos(\x)}, {1.5*sin(\x)});

\node at (1.5,.3) {\small $s^{N-2}$};
\node at (.25,.2) {\small $s$};
\node at (-.25,.2) {\small $s$};
\node at (-1.5,.3) {\small $s^{N-2}$};
\node at (1,1.2) {\small $\gamma_k$};
\node at (-1,1.2) {\small $\gamma$};
\node at (0,-1.8) {\small $\gamma_{k+1}$};
\node at (0,-.3) {\small $v'$};
\end{tikzpicture}
\end{center}
\vspace{-12pt}
\caption{The case that $\gamma\neq \gamma_k$. Left: The blocks of $\gamma$ and $\gamma_k$ containing $v'$ are labelled by distinct elements $s\neq s'$ of $S$. Right: The blocks are labelled by the same $s\in S$. }
\end{figure}

\begin{remark}
 A similar argument was used in \cite[Subsection 7.1]{BDM-thick} to construct non-relatively hyperbolic classical $C'(\frac{1}{6})$-groups. The statement of Proposition~\ref{prop:divgpsnotrh} can also be deduced from the fact that the divergence function of a non-trivially relatively hyperbolic group is at least exponential \cite{Si-metrrh}.
\end{remark}

We conclude by constructing new examples of non-relatively hyperbolic groups with non-degenerate hyperbolically embedded subgroups as defined in \cite{DGO}. A group $H$ is acylindrically hyperbolic if and only if $H$ contains a non-degenerate hyperbolically embedded subgroup \cite{Os-acyl}, i.e. this is another characterization of acylindrical hyperbolicity.

\begin{defi}[{\cite[Definition 4.25]{DGO}}]\label{definition:hyperbolically_embedded}
Let $G$ be a group and $H$ a subgroup. Then $H$ is \emph{hyperbolically embedded in $G$} if there exists a presentation $(X,R)$ of $G$ relative to $H$ with a linear relative Dehn function such that the elements of $R$ have uniformly bounded length and such that the set of letters from $H$ appearing in elements of $R$ is finite. $H$ is a \emph{non-degenerate} hyperbolically embedded subgroup if it is an infinite, proper hyperbolically embedded subgroup.
\end{defi}

\begin{thm}\label{thm:hyperbolically_embedded} Let $H$ be a finitely generated infinite group. Then there exists a finitely generated group $G$ such that $H$ is a non-degenerate hyperbolically embedded subgroup of $G$ and such that $G$ is not hyperbolic relative to any collection of proper subgroups. 
\end{thm}

\begin{proof} Let $S=\{s_1,s_2,\dots,s_k\}$ be a finite generating set of $H$, where each $s_i$ is non-trivial in $H$, and let $\Z$ be generated by the element $t$. Consider the quotient $G$ of $H*\Z$ by the normal closure of
$$R_0:=\{[s_i,t^n]^{6}\mid 1\leqslant i\leqslant k, n\in \N\}.$$

We denote $\Z=\{t_n\mid n\in \Z\}$ and $R:=R_1\sqcup R_2$, where $R_1=\{[s_i,t_n]^{6}\mid 1\leqslant i\leqslant k, n\in \N\}$ and $R_2=\{t_mt_nt_{m+n}^{-1}\mid m,n\in \N\}$. Then $(\Z,R)$ is a presentation of $G$ relative to $H$ as in Definition~\ref{definition:hyperbolically_embedded}, and $(\emptyset, R_1)$ is a presentation of $G$ relative to $\{H,\Z\}$.

By Theorem~\ref{thm:relatively_hyperbolic}, $(\emptyset,R_1)$ is a presentation of $G$ relative to $\{H,\Z\}$ with a linear relative Dehn function. Denote by $M_H$, respectively $M_\Z$, all elements of $M(H)$, respectively $M(\Z)$, that represent the identity in $H$, respectively $\Z$. Consider a diagram $D$ over $\langle H,\Z \mid M_H,M_\Z,R_1\rangle$. 
If a subdiagram $\Delta$ has a boundary word in $M(\Z)$, then there exists a diagram $\Delta'$ over $\langle \Z\mid R_2\rangle$ with the same boundary word as $\Delta$ such that $\Delta'$ has at most $|\partial\Delta|$ faces; $\Delta'$ is obtained by triangulating $\Delta$ as in Figure~\ref{figure:triangulating}. 
Moreover, any face with boundary word in $R_1$ contains exactly 6 edges with labels in $\Z$. Therefore, if $D$ has $n$ $R_1$-faces, then the sum of the boundary lengths of all maximal subdiagrams whose boundary words lie in $M(\Z)$ is at most $6n+|\partial D|$. Thus, if $D$ is a diagram over 
$\langle H,\Z\mid 
M_H,M_\Z,R_1\rangle$ with at most $n$ $R_1$-faces, then there exists a diagram over $\langle H,\Z\mid M_H,R_2,R_1\rangle$ with at most $7n+|\partial D|$ $R$-faces. Therefore, $(\Z,R)$ is a presentation of $G$ relative to $H$ that has a linear relative Dehn function, whence $H$ is hyperbolically embedded. It is non-degenerate since it is infinite and $G/\langle H\rangle^G\cong \Z$, whence $H\neq G$.

Remark~\ref{remark:convex_product} shows that each component of $\Gamma$ is isometrically embedded in $\Cay(G,S\cup\{t\})$. Using the same observations as in the proof of Proposition~\ref{prop:divgpsnotrh}, we can apply Proposition~\ref{notrh} to conclude that $G$ is not non-trivially relatively hyperbolic. 
\end{proof}
\begin{figure}\label{figure:triangulating}
\begin{center}
\begin{tikzpicture}[line width=1pt, >=stealth]
\begin{scope}[rotate=18]
\foreach \x in {1,2,3,4,5}{
\draw[dashed] (72*\x-72:2) to (72*\x:2);
%\node at (72*\x-36:1.2) {\small $t_{m_\x}$};
\fill (72*\x:2) circle (2pt);
\draw[rounded corners=5] (72*\x-72:2) to (72*\x-72:3) to (72*\x:3) to (72*\x:2);
}
\node at (0,0) {\small $\Delta$};
\end{scope}
\begin{scope}[shift={(7,0)},rotate=18]
\foreach \x in {1,2,3,4,5}{
\draw[dashed] (72*\x-72:2) to (72*\x:2);
\draw[dashed] (0,0) to (72*\x:2);
%\node at (72*\x-36:1.2) {\small $t_{m_\x}$};
\fill (72*\x:2) circle (2pt);
\draw[rounded corners=5] (72*\x-72:2) to (72*\x-72:3) to (72*\x:3) to (72*\x:2);
}
\fill (0,0) circle (2pt);
\end{scope}
\end{tikzpicture}
\end{center}
\vspace{-12pt}
\caption{Left: A subdiagram $\Delta$ in a diagram $D$ over $\langle H,\Z \mid M_H,M_\Z,R_1\rangle$ such that $\Delta$ has a boundary word in $M(\Z)$. The dashed lines represent edges labelled by elements of $\Z$. Right: We replace $\Delta$ by a diagram $\Delta'$ with at most $|\partial\Delta|$ faces all of which have labels in $R_2$.}
\end{figure}
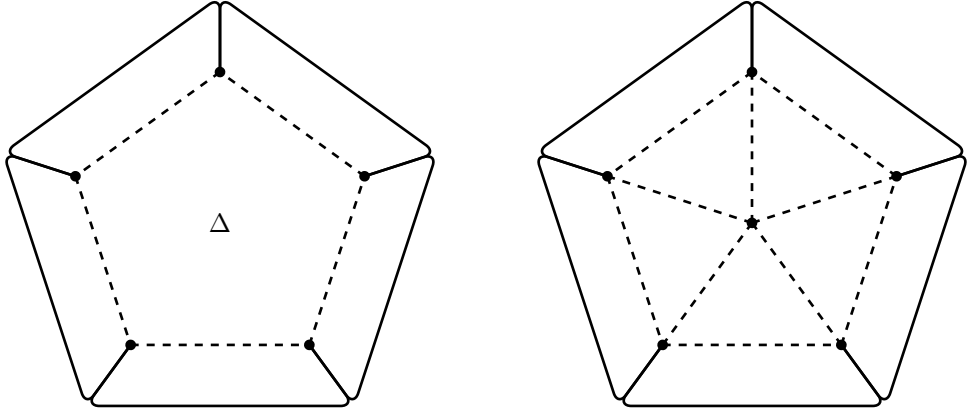

\begin{remark}
 Theorem~\ref{thm:hyperbolically_embedded} extends to any finite collection of finitely generated groups $\{G_1,G_2,\dots,G_l\}$. In the definition of $H$, one simply takes $G_1*G_2*\dots *G_l$ instead of $G$ and adapts the proof accordingly.
\end{remark}

\section*{Acknowledgements}
Prior to working on this paper, the first author learned that Denis Osin thought that infinitely presented classical $C'(\frac{1}{8})$-groups are acylindrically hyperbolic. 
This was inspired by results of Goulnara Arzhantseva and Cornelia Drutu on geodesics in the Cayley graphs of these groups \cite{ADr}. 
The present paper arose out of attempts to construct infinitely presented graphical $Gr(7)$-groups that are not acylindrically hyperbolic. 
We thank Goulnara Arzhantseva for her helpful comments, Denis Osin for his interest in our work, Markus Steenbock for an inspiring discussion on free product small cancellation, and an anonymous referee for a number of constructive remarks.
\bibliographystyle{alpha}
\bibliography{biblio}

\end{document}